\documentclass[11pt]{article}

\usepackage[a4paper,margin=1in]{geometry}
\usepackage{amsmath,amssymb,amsthm,mathtools}
\usepackage{booktabs}
\usepackage{array}
\usepackage{enumitem}
\usepackage{microtype}
\usepackage{graphicx}
\usepackage[colorlinks=true,linkcolor=blue,citecolor=blue,urlcolor=blue]{hyperref}

\numberwithin{equation}{section}

\newtheorem{theorem}{Theorem}[section]
\newtheorem{proposition}[theorem]{Proposition}
\newtheorem{lemma}[theorem]{Lemma}
\newtheorem{corollary}[theorem]{Corollary}

\theoremstyle{definition}
\newtheorem{definition}[theorem]{Definition}
\newtheorem{remark}[theorem]{Remark}
\newtheorem{example}[theorem]{Example}

\newtheorem{problem}[theorem]{Problem}

\newcommand{\PP}{\mathbb P}
\newcommand{\CC}{\mathbb C}
\newcommand{\QQ}{\mathbb Q}
\newcommand{\ZZ}{\mathbb Z}
\newcommand{\FF}{\mathbb F}
\newcommand{\OO}{\mathcal O}
\newcommand{\Wr}{\operatorname{Wr}}
\newcommand{\Aut}{\operatorname{Aut}}
\newcommand{\Hess}{\operatorname{Hess}}
\newcommand{\Span}{\operatorname{Span}}
\newcommand{\ord}{\operatorname{ord}}
\newcommand{\divisor}{\operatorname{div}}
\newcommand{\V}{\operatorname{V}}
\newcommand{\PSL}{\operatorname{PSL}}
\newcommand{\SL}{\operatorname{SL}}
\newcommand{\GL}{\operatorname{GL}}

\newcommand{\Klein}{\mathcal K}
\newcommand{\Fermat}{\mathcal F}
\newcommand{\KK}{\mathcal{Q}}
\newcommand{\Pic}{\mathcal{P}}
\newcommand{\W}{\mathcal W}
\newcommand{\sing}[1]{\texttt{#1}}

\title{Third-Order Weierstrass Points on Plane Quartics}
\author{Ewelina Nawara}
\date{\today}

\begin{document}
\maketitle

\begin{abstract}
For a smooth plane quartic $C$, the $3$-Weierstrass points are the $300$ points (counted with weights) at which some plane cubic has contact order at least ten with $C$. Revisiting the work of Kamel and Elshareef, we write the third Veronese Wronskian
as an $\SL_3$-covariant form $\Omega_3$ of degree $75$, given by a determinant of
iterated Hamiltonian derivatives, and explain the special $3$-Weierstrass points through contact cubics and homologies preserving $C$. We determine $\Omega_3$ for the Fermat, Komiya--Kuribayashi, Klein and Picard quartics; for the Klein quartic
$\Omega_3\equiv\Phi_6^2\,\Phi_{21}\,(3\Phi_{14}^3+21962\,\Phi_6^7)$ modulo the curve, up to a non-zero scalar.
Finally, the dectactic points of the Fermat quartic are not $4$-Weierstrass points, which disproves a conjecture of Kamel and Elshareef. All computations are supported by \texttt{SINGULAR} scripts.
\end{abstract}
 
\section{Introduction}\label{sec:intro}

Let $C\subset\PP^2$ be a smooth plane curve over $\CC$. Its \emph{flexes} are the points where the tangent line has contact order at least $3$, and these points are cut out by the Hessian. Its \emph{sextactic points} are the non-flexes where the osculating conic has contact order at least $6$. Cayley \cite{Cayley} gave a covariant of degree $12d-27$, the $2$-Hessian, cutting them out, and Maugesten and Moe \cite{MaugestenMoe} corrected Cayley's formula and showed, for rational curves, that sextactic points are detected by a Wronskian of the second Veronese embedding $\nu_2:\PP^2\to\PP^5$.

The next step is to look for points at which a plane cubic has contact order at least $10=\dim H^0(\PP^2,\OO(3))$ with $C$. The third Veronese map
$\nu_3:\PP^2\to\PP^9$ turns plane cubics into hyperplanes, so these are exactly the hyperosculation points of $\nu_3(C)\subset\PP^9$, and they are the zeros of a Wronskian, see Section~\ref{sec:wronskian}. For a smooth plane quartic the picture is especially clean: $C$ is canonically embedded, $\OO_C(3)\simeq 3K_C$, and these points are the classical Weierstrass points of order three, i.e.\ the Weierstrass points of the complete linear system $|3K_C|$, see \cite{FarkasKra,Laksov}. Their weighted number is
$300$, see Proposition~\ref{prop:degrees}.

\paragraph{The work of Kamel and Elshareef.} Third-order Weierstrass points of smooth quartics were studied by Kamel and Elshareef \cite{KamelElshareef}. Their approach is local and computational. First, they classify the $3$-gap sequences that can occur at flexes, at sextactic points and at the \emph{dectactic} points (the remaining $3$-Weierstrass points), and they show by explicit equations that every sequence in their list occurs \cite[Tables~1--3]{KamelElshareef}. Secondly, they propose a technique for computing the $3$-Weierstrass divisor: compute the Wronskian $W_3(x,y)$ of the affine monomials $1,x,y,\dots,y^3$ with respect to $x$ and then assemble
$\divisor(\Omega_3)=30\divisor(Z)+\divisor(W_3)+45\divisor(dx)$ from the divisors of $x$, $y$, $dx$ and of the line at infinity. They apply it to the Fermat and the Picard quartics, and they conjecture that $\W_q(\Fermat)\subseteq\W_{q+1}(\Fermat)$ for the Fermat quartic $\Fermat$ and all $q\ge1$
\cite[Section~5]{KamelElshareef}.

\paragraph{What is different here.} The present paper complements \cite{KamelElshareef} in three ways.\begin{enumerate}[label=(\arabic*)]
\item \emph{A global, covariant Wronskian.} We work with a single form $\Omega_n\in\CC[X,Y,Z]_{d_n}$, unique modulo the equation $F$ of $C$, whose restriction to $C$ has divisor $\W_n(C)$. It is given without denominators by $\Omega_n(x,y,1)\equiv\det\bigl(\delta^ib_j\bigr)$ with the Hamiltonian derivation $\delta=f_y\partial_x-f_x\partial_y$, and projectively by $\det(J_L^iM_j)\equiv (-1)^\frac{r(r+1)}{2}L^\frac{r(r+1)/2}{2}\,\Omega_n$ (the sign is $-1$ for $n\le4$), where $J_L(G)=\det(\nabla L,\nabla F,\nabla G)$, see Propositions~\ref{prop:hamilton} and~\ref{prop:projective}. Hence $\Omega_n$ is $\SL_3$-covariant by Corollary~\ref{cor:covariance}, a factorisation of $\Omega_3$ into forms is the same thing as a decomposition of $\W_3(C)$, and no bookkeeping at infinity or at the ramification points of a projection is needed.
\item \emph{Geometry instead of case analysis.} We give a uniform geometric proof of the exclusions in the classification of \cite{KamelElshareef} via linear equivalence on a contact cubic, Proposition~\ref{prop:no11}, characterise the ordinary flexes of $3$-weight four as those $P$ for which $F=\lambda t_P^4+\ell\cdot e$ with $T_P=\V(t_P)$ and a cubic $E=\V(e)$ inflectional at $P$, see Proposition~\ref{prop:totalcontact}, and we show how homologies preserving $C$ force weights in Theorem~\ref{thm:homologies}: axis points of involutions are sextactic and $3$-Weierstrass (or hyperflexes), axis points of homologies of order three are ordinary flexes of $3$-weight four, axis
points of homologies of order four are hyperflexes, and all of them lie in $\W_q(C)$ for every $q\ge2$, see Theorem~\ref{thm:allq}. \item \emph{More examples, and order four.} We treat the Fermat, Komiya--Kuribayashi and Klein quartics and a general quartic, and we re-derive the computation of \cite{KamelElshareef} for the Picard quartic, where now every flex and every sextactic point in $\W_3$ is explained by a homology. On the Fermat quartic we compute $\Omega_4$ and find that none of the $192$ dectactic points is a $4$-Weierstrass point, see Theorem~\ref{thm:fermat4}, so the conjecture of \cite{KamelElshareef} fails for $q=3$, although $\W_1(\Fermat)\cup\W_2(\Fermat)\subseteq\W_q(\Fermat)$ for all $q\ge2$.
\end{enumerate}
The results on the $3$-Weierstrass divisors are summarised in Table~\ref{tab:summary} below.

\begin{table}[ht]
\centering
\small
\renewcommand{\arraystretch}{1.25}
\begin{tabular}{@{}>{\raggedright}p{2.6cm}>{\raggedright}p{3.6cm}>{\raggedright}p{2.9cm}>{\raggedright}p{2.2cm}c@{}}
\toprule
curve & flexes in $\W_3$ & sextactic points in $\W_3$ & dectactic points & distinct\tabularnewline
\midrule
Fermat $\Fermat$ & $12$ hyperflexes, weight $5$ & all $48$, weight $1$ & $192$, weight $1$ & $252$\tabularnewline Komiya--Kuribayashi $\KK$ & $12$ hyperflexes, weight $5$ & $24$ of the $48$, weight $1$ & $216$, weight $1$ & $252$\tabularnewline
Klein $\Klein$ & $24$ ordinary flexes, weight $2$ & all $84$, weight $1$ & $168$, weight $1$ & $276$\tabularnewline
Picard $\Pic$ \cite{KamelElshareef} & $4$ hyperflexes, weight $5$; $16$ ordinary flexes, weight $4$ & $24$ of the $72$, weight $1$ & $192$, weight $1$ & $236$\tabularnewline
general quartic & none & none & $300$, weight $1$ & $300$\tabularnewline
\bottomrule
\end{tabular}
\caption{The $3$-Weierstrass divisor $\W_3$ of the quartics studied here. In each row the weighted total is $300$. In every example, the sextactic points in $\W_3$ are exactly those lying on axes of involutions, and the flexes in $\W_3$ are the hyperflexes and the points on axes of homologies of order three ---except for the Klein quartic, whose flexes have weight two for a non-symmetric reason.}
\label{tab:summary}
\end{table}

For the Klein quartic $\Klein:\ \Phi_4=X^3Y+Y^3Z+Z^3X=0$, with Klein's invariants
$\Phi_6,\Phi_{14},\Phi_{21}$, we prove in Theorem~\ref{thm:klein} that
\begin{equation}\label{eq:klein-intro}
  \Omega_3\equiv c\,\Phi_6^{2}\,\Phi_{21}\,\bigl(3\Phi_{14}^{3}+21962\,\Phi_6^{7}\bigr)
  \pmod{\Phi_4},\qquad c\in\CC^*.
\end{equation}
The invariant $\Phi_6^2\Phi_{21}^3$, which also has degree $75$ and would also give total weight $300$, \textbf{is not the right answer}, see Remark~\ref{rem:naive}. Here the symmetry explains the $84$ sextactic points of weight one, but not the weight two at the flexes: by \cite{KamelElshareef} an ordinary flex has $3$-weight $0$, $2$ or $4$,
and on $\Klein$ the value $2$ is realised by the cubic $XZ^2+Y^3$ at $(0:0:1)$.

\paragraph{Computations.} Every computational statement in this paper is checked by the \texttt{SINGULAR} \cite{Singular} scripts described in Appendix~\ref{app:comp}: the file \sing{omega.sing} implements the Hamiltonian formula for $\Omega_n$, tests for proportionality modulo $F$ and for transversality, and a routine for vanishing sequences; the scripts
\sing{general.sing}, \sing{fermat.sing}, \sing{kk.sing}, \sing{klein.sing},
\sing{picard.sing} and \sing{generic.sing} verify the statements of the corresponding sections, and \sing{groups.sing} those on automorphism groups and their homologies.

\paragraph{Organisation.} Section~\ref{sec:prelim} recalls Weierstrass points of linear series. Section~\ref{sec:wronskian} treats the third-Veronese Wronskian, for rational curves and then, via the Hamiltonian and projective formulas, for smooth quartics. Section~\ref{sec:local} contains the local analysis: monomial orders, contact cubics and the classification of \cite{KamelElshareef}, and homologies. Sections~\ref{sec:fermat}--\ref{sec:picard} treat the Fermat, Komiya--Kuribayashi, Klein and Picard quartics, and Section~\ref{sec:comparison} compares them and lists
open questions.

\section{Weierstrass points of the systems \texorpdfstring{$|\OO_C(n)|$}{|O(n)|}}\label{sec:prelim}

\subsection{Linear series}

Let $C$ be a smooth projective curve of genus $g$ over $\CC$, $L$ a line bundle of degree $e$ on $C$ and $V\subseteq H^0(C,L)$ a subspace of dimension $r+1$. For $P\in C$, the orders $\ord_P(s)$ of the non-zero sections $s\in V$ take exactly $r+1$
distinct values \[ a_0(P)<a_1(P)<\cdots<a_r(P),\]
the \emph{vanishing sequence} of $V$ at $P$. In characteristic zero it equals $0,1,\ldots,r$ at all but finitely many points. The \emph{weight} of $P$ is
\[ w_V(P)=\sum_{i=0}^{r}\bigl(a_i(P)-i\bigr)\ \ge 0,\] and $P$ is a \emph{Weierstrass point} of $V$ if $w_V(P)>0$, i.e.\ if some non-zero section of $V$ vanishes to order at least $r+1$ at $P$. In terms of the \emph{gap sequence} $n_i=a_{i-1}(P)+1$ $(1\le i\le r+1)$ the weight is $\sum_i(n_i-i)$ -- this is
the language presented used in \cite{FarkasKra,KamelElshareef}.

If $s_0,\ldots,s_r$ is a basis of $V$, $t$ a local coordinate and the $s_j$ are written as functions in a local frame of $L$, the Wronskian $\det\bigl(d^is_j/dt^i\bigr)_{0\le i,j\le r}$ glues to a non-zero section
\[ \Wr(V)\in H^0\bigl(C,\,L^{\otimes(r+1)}\otimes K_C^{\otimes r(r+1)/2}\bigr),\]
well-defined up to a non-zero scalar, with $\ord_P\Wr(V)=w_V(P)$ for every $P$. Taking degrees gives the Pl\"ucker (or Brill--Segre) formula \cite{ACGH,Laksov}:
\begin{equation}\label{eq:plucker}
  \sum_{P\in C}w_V(P)=(r+1)e+r(r+1)(g-1).
\end{equation}
The divisor $\W(V)=\sum_P w_V(P)P$ is the \emph{Weierstrass divisor} of $V$.

\subsection{Smooth plane quartics}

From now on, unless stated otherwise, $C=\V(F)\subset\PP^2$ is a smooth plane quartic. Then $g(C)=3$ and, by adjunction, $K_C\simeq\OO_C(1)$. Let $S_n=\CC[X,Y,Z]_n$. Plane
curves are projectively normal, so restriction $S_n\to H^0(C,\OO_C(n))$ is surjective
with kernel $F\cdot S_{n-4}$. We write
\[ V_n=H^0(C,\OO_C(n))=H^0(C,nK_C)=S_n/F\cdot S_{n-4},\qquad
  N_n=\dim V_n=\begin{cases}3,&n=1,\\ 4n-2,&n\ge2,\end{cases}\]
so that $N_n=\binom{n+2}{2}$ for $n\le3$, and $r_n=N_n-1$. We write $a^{(n)}_i(P)$, $w_n(P)$ and $\W_n(C)$ for the vanishing sequence, the weight and the Weierstrass divisor of $V_n$, and call the points of $\operatorname{supp}\W_n(C)$ the \emph{$n$-Weierstrass points} of $C$. Via $V_n\simeq H^0(C,(\Omega^1_C)^{\otimes n})$ these are the classical Weierstrass points of holomorphic $n$-differentials, which is the definition used in \cite{FarkasKra,KamelElshareef}.

\begin{proposition}\label{prop:degrees}
For every smooth plane quartic, $\deg\W_1(C)=24$, $\deg\W_2(C)=108$, $\deg\W_3(C)=300$ and $\deg\W_4(C)=588$.
\end{proposition}

\begin{proof}
Apply \eqref{eq:plucker} with $g=3$ and $(r,e)=(2,4),(5,8),(9,12),(13,16)$, respectively.
\end{proof}

For $n\ge 2$ these numbers agree with the classical count $(2n-1)^2g(g-1)^2$ of $n$-fold Weierstrass points \cite[Ch.~III]{FarkasKra}.

\begin{proposition}\label{prop:contact}
Let $P\in C$ and $n\ge1$. The vanishing sequence of $V_n$ at $P$ is the set of intersection multiplicities $I_P(C,D)$, where $D$ runs over the plane curves of degree $n$ not containing $C$. Consequently $P$ is an $n$-Weierstrass point if and
only if there is such a curve $D$ with $I_P(C,D)\ge N_n$; for $n=1,2,3$ this means $I_P(C,D)\ge 3$, $6$, $10$, respectively.
\end{proposition}

\begin{proof}
The non-zero sections of $V_n$ are the restrictions of the forms $G\in S_n\setminus F\cdot S_{n-4}$, and since $P$ is a smooth point of $C$, $\ord_P(G|_C)=I_P(C,\V(G))$. A section of order $\ge N_n=r+1$ exists if and only if the vanishing sequence differs from $0,1,\ldots,r$.
\end{proof}

\begin{definition}\label{def:types}
Let $P\in C$ and let $T_P$ be the tangent line.
\begin{enumerate}[label=(\alph*)]
\item $P$ is a \emph{flex} if $I_P(C,T_P)\ge 3$; an \emph{ordinary flex} if $I_P(C,T_P)=3$ and a \emph{hyperflex} if $I_P(C,T_P)=4$. \item If $P$ is not a flex, there is a unique conic $D_P$, the \emph{osculating conic}, with $I_P(C,D_P)\ge5$; it is irreducible. $P$ is \emph{sextactic} if $I_P(C,D_P)\ge 6$, and \emph{$s$-sextactic} if $I_P(C,D_P)=5+s$ with $s\in\{1,2,3\}$.
\item $P$ is \emph{dectactic} if it is neither a flex nor sextactic and some cubic $E$ has $I_P(C,E)\ge 10$. Such a cubic $\Sigma_P$ is unique and $P$ is \emph{$t$-dectactic} if $I_P(C,\Sigma_P)=9+t$ with $t\in\{1,2,3\}$.
\end{enumerate}
\end{definition}

The terminology in (b) and (c) is that of \cite{KamelElshareef}. The cubic in (c) is automatically irreducible: at a point that is neither a flex nor sextactic, a line has
contact at most $2$ and a conic at most $5$, so a reducible cubic has contact at most $7$. It is also unique: if $E\neq E'$ were two such cubics, both irreducible, the
ultrametric inequality applied to their branches at $P$ would give
$I_P(E,E')\ge 10>9$, contrary to B\'ezout's theorem. The same count shows
\cite[Lemma~7]{KamelElshareef} that a reducible cubic with contact $\ge 10$ exists only at a hyperflex ($T_P^3$, contact $12$) or at a $3$-sextactic point ($D_P\cdot T_P$, contact $10$). By Proposition~\ref{prop:contact}, $\W_1(C)$ is supported on the flexes, $\W_2(C)$ on the flexes and the sextactic points (every flex lies in $\W_2(C)$ by Lemmas~\ref{lem:hyperflex} and~\ref{lem:flex} below), and
\[\operatorname{supp}\W_3(C)\subseteq
  \{\textup{flexes}\}\cup\{\textup{sextactic points}\}\cup\{\textup{dectactic points}\},\]
where the dectactic points are exactly the $3$-Weierstrass points that are neither flexes nor sextactic. Whether a given flex or sextactic point lies in $\W_3(C)$ is a genuine question, answered locally in Section~\ref{sec:local}.

\section{The third-Veronese Wronskian}\label{sec:wronskian}

Order the ten cubic monomials as
\[ \begin{array}{lllll}
M_0=X^3, & M_1=X^2Y, & M_2=X^2Z, & M_3=XY^2, &
M_4=XYZ, \\ M_5=XZ^2, & M_6=Y^3, & M_7=Y^2Z,&
M_8=YZ^2, & M_9=Z^3,
\end{array}\] and let $\nu_3 (M_0:\cdots:M_9):\PP^2\to\PP^9$ be the third Veronese embedding. The hyperplane $\sum_i\lambda_iT_i=0$ of $\PP^9$ pulls back to the cubic $\sum_i\lambda_iM_i=0$, and for a curve $C\subset\PP^2$ the contact order of a cubic
$E$ with a branch of $C$ equals the contact order of the corresponding hyperplane with the corresponding branch of $\nu_3(C)$. Thus third-order Weierstrass points of $C$ are the Weierstrass points (hyperosculation points) of $\nu_3(C)\subset\PP^9$ with respect to hyperplanes, exactly as sextactic points are related to $\nu_2(C)\subset\PP^5$ in \cite{MaugestenMoe}.

\subsection{Rational curves}

Let $\varphi(\varphi_0:\varphi_1:\varphi_2):\PP^1\to C\subset\PP^2$ be a birational parametrisation of an irreducible plane curve of degree $d$, where the $\varphi_i$ are binary forms of degree $d$ without common zero, and put
\[ F_i(s,t)=M_i\bigl(\varphi_0(s,t),\varphi_1(s,t),\varphi_2(s,t)\bigr),\qquad 0\le i\le 9,\]
binary forms of degree $3d$.

\begin{definition}
The \emph{third-Veronese Wronskian} of $\varphi$ is the binary form
\[ \Xi_3(s,t)=\det\Bigl(\frac{\partial^9F_i}{\partial s^{9-j}\,\partial t^{j}}\Bigr)_{0\le i,j\le 9}.\]
Up to a non-zero constant it is the homogenisation of degree $10(3d-9)$ of the ordinary Wronskian $\det\bigl(\frac{d^j}{du^j}F_i(u,1)\bigr)_{0\le i,j\le 9}$.
\end{definition}

\begin{theorem}
\label{thm:rational}
Let $C$ be an irreducible rational plane curve of degree $d\ge 4$ with birational parametrisation $\varphi$. Then $F_0,\ldots,F_9$ are linearly independent, $\Xi_3$ is a non-zero form of degree
\[ 10(3d-9)=30d-90, \]
and for every $q=(s_0:t_0)\in\PP^1$ the order of vanishing of $\Xi_3$ at $q$ equals the weight at $q$ of the linear series $\Span(F_0,\dots,F_9)$. In particular, the following are equivalent:
\begin{enumerate}[label=\rm(\roman*)]
\item $\Xi_3(s_0,t_0)=0$;
\item there is a plane cubic $E$ whose contact order with the branch of $C$ at $\varphi(q)$, i.e.\ $\ord_q\bigl(E\circ\varphi\bigr)$, is at least $10$.
\end{enumerate}
If $\varphi(q)$ is a smooth point of $C$, the contact order in (ii) is the intersection multiplicity $I_{\varphi(q)}(C,E)$.
\end{theorem}

\begin{proof}
A cubic form vanishing identically on $C$ would contain $C$, which is impossible for $d\ge 4$; hence the $F_i$ are independent. They span a base-point-free linear series $g^9_{3d}$ on $\PP^1$ (the cubics have no common zero), namely the pull-back of the
hyperplane series of $\nu_3$. Its Wronskian is a non-zero section of $\OO_{\PP^1}(10\cdot 3d)\otimes K_{\PP^1}^{\otimes 45}=\OO_{\PP^1}(30d-90)$, in
accordance with \eqref{eq:plucker} for $g=0$, and it vanishes at $q$ to order equal to the weight of the series at $q$; the non-zero sections are the forms $E\circ\varphi$. The completeness of the series is not needed.
\end{proof}

\begin{remark}
For $d\le 3$ the curve lies on a cubic, the $F_i$ are dependent and $\Xi_3\equiv 0$; one must then work with the restricted series rather than with all plane cubics. For $d\ge 4$ the independence hypothesis is automatic.
\end{remark}

\subsection{Smooth quartics: affine, Hamiltonian and projective formulas}

For a smooth quartic $C=\V(F)$, $\W_n(C)$ is the divisor of $\Wr(V_n)\in H^0(C,\OO_C(n)^{\otimes N_n}\otimes K_C^{\otimes r_n(r_n+1)/2})=H^0(C,\OO_C(d_n))$ with \[ d_n=nN_n+\tfrac{r_n(r_n+1)}{2}=6,\ 27,\ 75,\ 147\qquad(n=1,2,3,4), \] and $4d_n=\deg\W_n(C)$. Throughout, $M_0,\dots,M_{r}$ ($r=r_n$) denotes a basis of $V_n$ consisting of monomials of degree $n$. For $n\le3$ all of them; for $n=4$ all quartic monomials but one, e.g.\ the standard monomials with respect to a Gr\"obner basis of $(F)$. Let $b_j=M_j(x,y,1)$.

\begin{lemma}\label{lem:affine}
Let $f(x,y)=F(x,y,1)$ and let $D=\partial_x-(f_x/f_y)\,\partial_y$ be differentiation with respect to $x$ along $C$. We put
\[ W_x^{(n)}=\det\bigl(D^ib_j\bigr)_{0\le i,j\le r_n}.\]
There is a form $\Omega_n\in S_{d_n}$, unique modulo $F\cdot S_{d_n-4}$ and up to a non-zero scalar, such that on the affine part of $C$
\begin{equation}\label{eq:affine}
  \Omega_n(x,y,1)=W_x^{(n)}\cdot f_y^{\,r_n(r_n+1)/2},
\end{equation}
and $\divisor(\Omega_n|_C)=\W_n(C)$. For $n=1$, $\Omega_1$ is a non-zero multiple of the Hessian.
\end{lemma}

\begin{proof}
On $U=C\cap\{Z\ne 0\}$ use the frame $Z^n$ of $\OO_C(n)$, so that the sections of $V_n$ become the functions $b_j$. The differential $\omega_0=dx/f_y$ is regular and nowhere zero on $U$ (where $f_y=0$ it equals $-dy/f_x$), and under the adjunction isomorphism $K_C\simeq\OO_C(1)$ it corresponds to the section $Z$: the canonical differentials of $C$ are $G(x,y,1)\,dx/f_y$ with $G$ linear. On $U$ the
Wronskian is $W^{(n)}_x\,(dx)^{\otimes r(r+1)/2}\otimes(Z^n)^{\otimes(r+1)}$ with $r=r_n$, and $dx=f_y\,\omega_0$. Hence, in the frame $Z^{d_n}$ of $\OO_C(d_n)$, $\Wr(V_n)$ is represented by $W^{(n)}_xf_y^{\,r(r+1)/2}$. By projective normality this global section is the restriction of a form $\Omega_n$ of degree $d_n$, unique modulo $F$.
For $n=1$ both $\Omega_1$ and the Hessian are forms of degree $6$ whose restrictions have divisor equal to the flex divisor, so they agree modulo $F$ up to a scalar.
\end{proof}

The rational functions $D^ib_j$ have powers of $f_y$ in their denominators. They disappear if we differentiate along the \emph{Hamiltonian vector field} of $f$.

\begin{proposition}
\label{prop:hamilton}
Let $\delta=f_y\,\partial_x-f_x\,\partial_y$. Then $\delta(f)=0$, so $\delta$ is a derivation of $\CC[x,y]/(f)$, and
\[\Omega_n(x,y,1)\equiv\det\bigl(\delta^ib_j\bigr)_{0\le i,j\le r_n}\pmod f .\]
\end{proposition}

\begin{proof}
As derivations of the function field of $C$ we have $\delta=f_y D$. By induction, $\delta^i=f_y^iD^i+\sum_{k<i}c_{ik}D^k$ with rational functions $c_{ik}$ that do not depend on the function to which the operator is applied. Subtracting multiples of earlier rows we get $\det(\delta^ib_j)=f_y^{0+1+\dots+r}\det(D^ib_j)=f_y^{r(r+1)/2}W^{(n)}_x$, and \eqref{eq:affine} applies.
\end{proof}

The derivation $\delta$ is the dehomogenisation of a projective one. For a linear form $L$ and $G\in S_m$ put
\[ J_L(G)=\det\begin{pmatrix}\nabla L\\ \nabla F\\ \nabla G\end{pmatrix}\in S_{m+2},\]
a derivation of $S$ with $J_L(F)=0$, so $J_L$ preserves the ideal $(F)$.

\begin{proposition}
\label{prop:projective}
For every linear form $L\neq0$ and every $n\ge1$,
\[ \det\bigl(J_L^{\,i}M_j\bigr)_{0\le i,j\le r_n}\equiv(-1)^{r_n(r_n+1)/2}\,L^{\,r_n(r_n+1)/2}\,\Omega_n \pmod F, \]
where $\Omega_n$ is normalised by Proposition~\ref{prop:hamilton}. For $n=1,2,3,4$ the
sign is $-1$.
\end{proposition}

\begin{proof}
Let $G\in S_m$. The $4\times4$ matrix with rows $(Z,\nabla Z)$, $(L,\nabla L)$, $(4F,\nabla F)$, $(mG,\nabla G)$ has vanishing determinant, because by Euler's formula its first column is $X$ times the second plus $Y$ times the third plus $Z$ times the
fourth. Expanding along the first column and using
$\det(\nabla Z,\nabla L,\nabla F)=-J_Z(L)$ we get
\begin{equation}\label{eq:euler} Z\,J_L(G)\equiv L\,J_Z(G)-m\,J_Z(L)\,G\pmod F .
\end{equation}
Thus on $C$ the operator $J_L$ acts on forms of degree $m$ as $(L/Z)J_Z-m\,J_Z(L)/Z$. All $M_j$ have degree $n$ and $J_Z^kM_j$ has degree $n+2k$, so
by induction
\[ J_L^iM_j\equiv(L/Z)^iJ_Z^iM_j+\textstyle\sum_{k<i}c_{ik}J_Z^kM_j\quad\text{on }C,\]
with rational functions $c_{ik}$ independent of $j$, and therefore $\det(J_L^iM_j)\equiv(L/Z)^{r(r+1)/2}\det(J_Z^iM_j)$ on $C$. Finally \[ J_Z(G)(x,y,1)=(F_XG_Y-F_YG_X)(x,y,1)=-\delta\bigl(G(x,y,1)\bigr),\]
so $\det(J_Z^iM_j)(x,y,1)=(-1)^{r(r+1)/2}\det(\delta^ib_j)$. Put $\varepsilon=r(r+1)/2$.
The two forms $\det(J_Z^iM_j)$ and $(-1)^{\varepsilon}Z^{\varepsilon}\Omega_n$, both of
degree $d_n+\varepsilon$, agree on $C\cap\{Z\neq0\}$ by Proposition~\ref{prop:hamilton}, hence modulo $F$. For $n=1,2,3,4$ we have $r(r+1)/2=3,15,45,91$.
\end{proof}

\begin{corollary}
\label{cor:covariance}
Let $n\le 3$, $g\in\GL_3(\CC)$, and write $\Omega_n(F)$ for the form attached to $F$.
Then
\[ \Omega_n(F\circ g)\equiv\det(g)^{e_n}\,\Omega_n(F)\circ g\pmod{F\circ g},\qquad
  e_n=\tfrac{r_n(r_n+1)}{2}-\tfrac{nN_n}{3}\in\{2,11,35\}, \]
and $\Omega_n(\lambda F)=\lambda^{r_n(r_n+1)/2}\Omega_n(F)$. In particular, if $g\in\SL_3(\CC)$ and $F\circ g=F$, then $\Omega_n\circ g\equiv\Omega_n\pmod F$. If
$F\circ g=\lambda F$, then $\Omega_3\circ g\equiv\lambda^{45}\det(g)^{-35}\Omega_3$.
\end{corollary}

\begin{proof}
By the chain rule, $J^{F\circ g}_{L\circ g}(G\circ g)=\det(g)\,(J^F_LG)\circ g$. The monomials transform as $M_j\circ g=\sum_kA_{jk}M_k$ with
$\det A=\det(\operatorname{Sym}^ng)=\det(g)^{nN_n/3}$. Hence $\det A\cdot\det\bigl((J^{F\circ g}_{L\circ g})^iM_k\bigr)=\det(g)^{r(r+1)/2} \det(J_L^iM_j)\circ g$, and Proposition ~\ref{prop:projective}, applied to both sides,
gives the claim after cancelling $(L\circ g)^{r(r+1)/2}$. The scaling in $F$ follows from $\delta_{\lambda f}=\lambda\delta_f$.
\end{proof}

For $n=1$ this is the familiar covariance $\Hess(F\circ g)=\det(g)^2\Hess(F)\circ g$.
The script \sing{general.sing} compares $\det(J_L^iM_j)$ with $\Omega_n$ computed by Proposition~\ref{prop:hamilton}, which checks Proposition~\ref{prop:projective} for the quartic $C_0$ of Remark~\ref{rem:general} ($n=1,2$ over $\QQ$, $n=3$ modulo $32003$) and for the Klein quartic ($n=1,2$), and it checks Corollary~\ref{cor:covariance} on $C_0$ for a matrix in $\SL_3(\ZZ)$ and one of determinant $2$, and for $\lambda=3$.

\begin{remark}[Comparison with \cite{KamelElshareef}]\label{rem:KEcomparison}
Kamel and Elshareef compute $W_3=W^{(3)}_x$ and then the divisor $\divisor(\Omega_3)=30\divisor(Z)+\divisor(W_3)+45\divisor(dx)$, which requires the divisors of $x$, $y$, $dx$ and of the line at infinity and a separate discussion of the points at infinity \cite[Section~4]{KamelElshareef}. In our normalisation this is the
identity $\Omega_3(x,y,1)=W_3\,f_y^{45}$ of Lemma~\ref{lem:affine}: since
$dx=f_y\,\omega_0$ and $\omega_0\leftrightarrow Z$, the bookkeeping is absorbed into a single global form, whose factors can be read off directly. For instance, for the Fermat quartic the authors in \cite{KamelElshareef} find $W_3=c\,x^5P_{12}(x)P_{48}(x)/y^{130}$, i.e.\
$\Omega_3(x,y,1)=c'\,x^5y^5P_{12}(x)P_{48}(x)$, the dehomogenisation of $X^5Y^5Z^5\Delta_{12}\Pi_{48}$ in Theorem~\ref{thm:fermat}, and for the Picard quartic they find $W_3=c\,x^5(x^4+9)^4Q_8(x)Q_{64}(x)/y^{86}$, i.e.\
$\Omega_3(x,y,1)=c'x^5y^4(x^4+9)^4Q_8Q_{64}$, the dehomogenisation of the product in Theorem~\ref{thm:picard}.
\end{remark}

\begin{remark}\label{rem:recipe}
Besides the exact formula of Proposition~\ref{prop:hamilton} one can evaluate
\eqref{eq:affine} at points of $C$ using power series: if $x=x_0+t$ and $y=y(t)$ parametrise the branch through $(x_0,y_0)$, then $D^ib_j(x_0,y_0)=i!\,[t^i]\,b_j(x_0+t,y(t))$. The same coefficients give the vanishing sequence at $(x_0,y_0)$: it is the set of pivot columns of the matrix
$\bigl([t^k]\,b_j(x_0+t,y(t))\bigr)_{j,k}$. Equivalently, and without power series: for a form $G$ of degree $n$ we have $\ord_P(G|_C)\ge k$ if and only if $G(x,y,1)\in(f)+\mathfrak m_P^k$, so the number of vanishing orders smaller than $k$ is the rank of the restriction map $V_n\to\CC[x,y]/\bigl((f)+\mathfrak m_P^k\bigr)$. This is how the procedure \sing{vanseq} of \sing{omega.sing} works, over $\QQ$, over number fields and over finite fields.
\end{remark}

\section{Local analysis}\label{sec:local}
\subsection{Monomial orders}

Let $P\in C$. After a projective change of coordinates, $P$ is the origin of the affine chart $Z=1$ and $T_P=\{y=0\}$. Then $x$ is a local coordinate $t$ at $P$ and the branch is $y=\alpha t^k+O(t^{k+1})$ with $\alpha\neq 0$ and $k=I_P(C,T_P)\in\{2,3,4\}$. The affine monomials $x^iy^j$ with $i+j\le n$ restrict to a spanning set of $V_n$ and 
\begin{equation}\label{eq:monorders}
  \ord_P(x^iy^j)=i+kj .
\end{equation}
We use repeatedly that every order attained by a non-zero section belongs to the vanishing sequence, and that if basis elements have pairwise distinct orders, then these orders \emph{are} the vanishing sequence.

\begin{lemma}\label{lem:hyperflex}
If $P$ is a hyperflex ($k=4$), the vanishing sequences of $V_1,V_2,V_3$ at $P$ are
\[ \{0,1,4\},\qquad \{0,1,2,4,5,8\},\qquad \{0,1,2,3,4,5,6,8,9,12\}, \]
so $w_1(P)=2$, $w_2(P)=5$ and $w_3(P)=5$ on every smooth quartic. Moreover, we have
$P\in\W_q(C)$ for every $q\ge1$.
\end{lemma}

\begin{proof}
By \eqref{eq:monorders} the orders $i+4j$ for $i+j\le 3$ are $0,1,2,3$ $(j=0)$, $4,5,6$ $(j=1)$, $8,9$ $(j=2)$ and $12$ $(j=3)$, pairwise distinct.
Similarly, for $i+j\le 1,2$ the weights are $5-3$, $20-15$ and $50-45$. For the last claim, $T_P^q$ has order $4q\ge N_q$.
\end{proof}

\begin{lemma}\label{lem:flex}
Let $P$ be an ordinary flex ($k=3$). Then the vanishing sequences of $V_1$ and $V_2$
at $P$ are $\{0,1,3\}$ and $\{0,1,2,3,4,6\}$, so $w_1(P)=w_2(P)=1$. The vanishing sequence of $V_3$ at $P$ is
\[ \{0,1,2,3,4,5,6,7,9\}\cup\{m\},\qquad m=8\ \textup{ or }\ m\ge 10, \]
so either $m=8$ and $w_3(P)=0$, or $m\ge10$ and $w_3(P)=m-8\ge 2$. In particular, an ordinary flex is never a $3$-Weierstrass point of weight $1$.
\end{lemma}

\begin{proof}
For $n=2$ the orders $i+3j$ ($i+j\le 2$) are $0,1,2,3,4,6$, pairwise distinct. For $n=3$ the monomials $1,x,x^2,x^3,y,xy,x^2y,y^2,xy^2,y^3$ have orders $0,1,2,3,3,4,5,6,7,9$, so $\{0,\ldots,7,9\}$ lies in the vanishing sequence and the tenth element $m$ is not in this set. The weight is $(28+9+m)-45=m-8$ if $m\ge10$, and $0$ if $m=8$.
\end{proof}

\begin{lemma}\label{lem:nonflex}
If $P$ is not a flex ($k=2$), the vanishing sequence of $V_2$ at $P$ is $\{0,1,2,3,4,m\}$ with $m=I_P(C,D_P)\ge 5$. $P$ is sextactic if and only if $m\ge 6$, and then $w_2(P)=m-5$, i.e.\ $P$ is $w_2(P)$-sextactic.
\end{lemma}

\begin{proof}
The monomials $1,x,x^2,y,xy,y^2$ have orders $0,1,2,2,3,4$.
\end{proof}

Together with Lemma~\ref{lem:hyperflex} this shows that every flex is a $2$-Weierstrass point and that $\operatorname{supp}\W_2(C)$ is the set of flexes and
sextactic points.

\subsection{Contact cubics and the classification of Kamel--Elshareef}

The values of $m$ in Lemma~\ref{lem:flex}, and the analogous values at sextactic and dectactic points, were determined in \cite{KamelElshareef}. The upper bounds come from B\'ezout's theorem -- the only non-trivial exclusions rest on one geometric fact, which we isolate.

\begin{lemma}\label{lem:lineq}
Let $E\subset\PP^2$ be an integral cubic and let $P$ be a smooth point of $E$. Let $G$ and $G'$ be curves of degree $e$ not containing $E$ whose intersection divisors on $E$ are $E\cdot G=(3e-1)P+Q$ and $E\cdot G'=3e\,P$. Then $Q=P$.
\end{lemma}

\begin{proof}
If $Q\neq P$, then $Q$ is a smooth point of $E$ (at a singular point every curve meets $E$ with multiplicity at least $2$). Both divisors lie in $|\OO_E(e)|$, so $Q\sim P$. A rational function $\varphi$ on $E$ with $\divisor\varphi=Q-P$ is a unit at every singular point of $E$, hence defines a finite birational morphism $E\to\PP^1$, which is an isomorphism since $\PP^1$ is normal. This contradicts $p_a(E)=1$.
\end{proof}

We also use the elementary ultrametric inequality
$I_P(A,B)\ge\min\{I_P(A,C),I_P(C,B)\}$ for curves $A,B$ through a smooth point $P$ of $C$, applied to branches \cite[Lemma~5]{KamelElshareef}, and the fact that a cubic with a cusp at $P$ meets a smooth branch through $P$ with multiplicity at most $3$.

\begin{proposition}[{\cite[Lemmas~9, 11, 13]{KamelElshareef}}]\label{prop:no11}
Let $P$ be a point of a smooth quartic $C$.
\begin{enumerate}[label=\rm(\alph*)]
\item If $P$ is an ordinary flex, no cubic $E$ has $I_P(C,E)=11$.
\item If $P$ is $1$-sextactic, no cubic $E$ has $I_P(C,E)=11$.
\item If $P$ is dectactic and some cubic $R$ has $I_P(C,R)=9$, then no cubic $E$ has
  $I_P(C,E)=12$.
\end{enumerate}
\end{proposition}

\begin{proof}
(a) Let $I_P(C,E)=11$ and $T=T_P$. Lines meet $C$ at $P$ with multiplicity $\le3$ and conics with multiplicity $\le 6$ (Lemma~\ref{lem:flex}), so $E$ is irreducible. If $P$ were a node of $E$, one branch would be transversal to $C$ and the other one,
$\Gamma$, would have $I_P(\Gamma,C)=10$; then $I_P(\Gamma,T)\ge\min\{10,3\}=3$ and $I_P(E,T)\ge4$, so $T\subset E$, which is absurd. A cusp gives multiplicity $\le3$. Hence $E$ is smooth at $P$, $I_P(E,T)\ge\min\{11,3\}=3$, and $E\cdot T=3P$ by
B\'ezout. Now $E\cdot C=11P+Q$ and $E\cdot T^4=12P$, and Lemma~\ref{lem:lineq} with $e=4$ gives $Q=P$, i.e.\ $I_P(C,E)=12$, a contradiction.

(b) This is the same argument with the osculating conic $D_P$ in place of $T$: $I_P(E,D_P)\ge\min\{11,6\}=6$, so $E\cdot D_P=6P$ and $E\cdot D_P^2=12P$.

(c) Suppose that $I_P(C,E)=12$, and let $R$ be a cubic with $I_P(C,R)=9$. Since $P$ is neither a flex nor a sextactic point, any line has contact at most $2$ with $C$ at $P$, and any conic has contact at most $5$. Consequently, both $E$ and $R$ are irreducible, and $E\neq R$.

We claim that $E$ is smooth at $P$. A cusp is impossible, since a cuspidal cubic meets a smooth branch through its cusp with multiplicity at most $3$. Suppose that $E$ has a node at $P$, with branches $\Gamma_1$ and $\Gamma_2$. Since the two branches have distinct tangent lines and $C$ is smooth, we may assume that \[I_P(C,\Gamma_1)=1,\qquad I_P(C,\Gamma_2)=11.
\]
We have \(I_P(\Gamma_2,R)\ge 9\). If $R$ is smooth at $P$, this is the ultrametric
inequality \(I_P(\Gamma_2,R)\ge\min\{I_P(\Gamma_2,C),I_P(C,R)\}=9\). $R$ cannot have a cusp at $P$ (contact would be at most $3$). If $R$ has a node at $P$ with branches $R_1,R_2$, where $I_P(C,R_1)=1$ and $I_P(C,R_2)=8$, then
\(I_P(\Gamma_2,R)\ge I_P(\Gamma_2,R_1)+\min\{I_P(\Gamma_2,C),I_P(C,R_2)\}\ge 1+8=9\). It follows that \( I_P(E,R)=I_P(\Gamma_1,R)+I_P(\Gamma_2,R) \geq 1+9=10, \) contrary to B\'ezout's theorem, since $E$ and $R$ are distinct irreducible cubics. Thus $E$ is smooth at $P$.

Again by the ultrametric inequality, applied branchwise if $R$ is nodal at $P$, \( I_P(E,R)\geq 9 \) and hence $E\cdot R=9P$ by B\'ezout. Let $T=T_P$. Since $P$ is not a flex, $I_P(C,T)=2$, and therefore $I_P(E,T)\geq2$. We may write \( E\cdot T=2P+Q'. \)
Consequently, \[ E\cdot(RT)=11P+Q', \qquad E\cdot C=12P.\] Lemma~\ref{lem:lineq}, applied with $e=4$, gives $Q'=P$, so $I_P(E,T)=3$. But then \(I_P(C,T)\geq \min\{I_P(C,E),I_P(E,T)\}=3, \) contradicting the assumption that $P$ is not a flex.
\end{proof}

Combining these exclusions with Lemmas~\ref{lem:hyperflex}--\ref{lem:nonflex},
B\'ezout's theorem and explicit examples, Kamel and Elshareef obtained the following classification \cite[Propositions~10, 12, 14]{KamelElshareef}. Here we state it in terms of vanishing sequences ($a_i=n_{i+1}-1$).

\begin{theorem}[Kamel--Elshareef]\label{thm:KE}
The vanishing sequence of $V_3$ at a point $P$ of a smooth plane quartic is one of the following, and each of them occurs.
\[ \renewcommand{\arraystretch}{1.1}
\begin{array}{llc}
\toprule
\textup{type of }P & \textup{vanishing sequence of }V_3 & w_3(P)\\
\midrule
\textup{hyperflex} & 0,1,2,3,4,5,6,8,9,12 & 5\\
\textup{ordinary flex} & 0,\ldots,9\ \mid\ 0,\ldots,7,9,10\ \mid\ 0,\ldots,7,9,12 & 0\mid2\mid4\\
\textup{$1$-sextactic} & 0,\ldots,9\ \mid\ 0,\ldots,8,10\ \mid\ 0,\ldots,8,12 & 0\mid1\mid3\\
\textup{$2$-sextactic} & 0,\ldots,9 & 0\\
\textup{$3$-sextactic} & 0,\ldots,6,8,9,10 & 3\\
\textup{$t$-dectactic, no cubic with contact $9$} & 0,\ldots,8,9+t & t\\
\textup{$t$-dectactic, a cubic with contact $9$} & 0,\ldots,7,9,9+t\quad(t\le2) & t+1\\
\textup{other} & 0,\ldots,9 & 0\\
\bottomrule
\end{array} \]
\end{theorem}

In particular, $m\in\{8,10,12\}$ in Lemma~\ref{lem:flex}. All realisability examples of
\cite[Section~3]{KamelElshareef} (for sample values of their parameters) are checked in \sing{general.sing}, including their smoothness. For the extreme case at a flex there is a simple geometric characterisation.

\begin{proposition}
\label{prop:totalcontact}
Let $P$ be an ordinary flex of $C=\V(F)$ with tangent line $T=\V(t)$. The following are equivalent:
\begin{enumerate}[label=\rm(\roman*)]
\item $w_3(P)=4$;
\item there is a cubic $E$ with $E\cdot C=12P$;
\item $F=\lambda\,t^4+\ell\cdot e$ with $\lambda\in\CC^*$, a linear form $\ell$, and a cubic $E=\V(e)$ with $I_P(E,T)=3$.
\end{enumerate}
The cubic in (ii) and (iii) is then irreducible, smooth at $P$, and has an inflection point at $P$ with tangent $T$, i.e.,\ $C$ lies in the pencil spanned by the quadruple tangent line $T^4$ and the reducible quartic $\ell\cdot E$.
\end{proposition}

\begin{proof}
(i)$\Longleftrightarrow$(ii) is Theorem~\ref{thm:KE} (the element $12$ of the sequence is attained exactly by a cubic with $E\cdot C=12P$). \\
(ii)$\Longrightarrow$(iii): as in the proof of Proposition~\ref{prop:no11}(a), $E$ is irreducible, smooth at $P$, and $E\cdot T=3P$. On the integral curve $E$ the forms $F$ and $t^4$ cut out the same
Cartier divisor $12P$, so $F/t^4$ is a nowhere vanishing regular function on $E$, i.e.\ a constant $\lambda$; thus $e$ divides $F-\lambda t^4$. \\
(iii)$\Longrightarrow$(ii): $E\cdot C=E\cdot(\lambda t^4)=4\,(E\cdot T)=12P$.
\end{proof}

\begin{remark}
In the same way a $1$-sextactic point $P$ has $w_3(P)=3$ if and only if $F=\lambda\,d^2+\ell\cdot e$, where $d=0$ is the osculating conic and $E\cdot D_P=6P$.
The examples of \cite{KamelElshareef} for the cases $m=12$ are of exactly these shapes: $F=(YZ^2-X^3)(aX+bY+cZ)+\gamma Y^4$ and $F=(YZ^2-X^2Z-Y^3)(aX+bY+cZ)+d(YZ-X^2)^2$. Homologies of order three produce the first shape automatically, see Theorem~\ref{thm:homologies}(b).
\end{remark}

\subsection{Homologies}

A \emph{homology} of $\PP^2$ is a projective transformation of finite order $N\ge2$ fixing a line $\ell$ (the \emph{axis}) pointwise and one further point $O\notin\ell$ (the \emph{centre}). In suitable coordinates $(u:v:w)$ it is $\sigma_N:(u:v:w)\mapsto(\zeta u:v:w)$, $\zeta$ a primitive $N$-th root of unity, with $\ell=\{u=0\}$ and $O=(1:0:0)$. Every involution of $\PP^2$ is a homology (a
\emph{harmonic homology}).

\begin{proposition}\label{prop:homnormal}
Let a homology $\sigma$ of order $N$ preserve the smooth quartic $C=\V(F)$. Then $N\in\{2,3,4\}$ and, in coordinates as above,
\[ F=\begin{cases}
  F_4(v,w)+u^2F_2(v,w)+c\,u^4,\ c\neq0, & N=2,\\
  F_4(v,w)+u^3F_1(v,w),\ F_1\neq0, & N=3,\\
  F_4(v,w)+c\,u^4,\ c\neq0, & N=4,
  \end{cases} \]
with $F_i$ binary forms of degree $i$ and $F_4$ without multiple factors. Thus $C$ meets the axis in four distinct points $P$, the tangent line $T_P$ is the line $PO$, and $\ell$ is transversal to $C$ at $P$. The centre lies on $C$ if and only if $N=3$.
\end{proposition}

\begin{proof}
Write $F=\sum_{a=0}^4u^aF_{4-a}(v,w)$. Since $F\circ\sigma=\zeta^{a}F$ for every $a$ with $F_{4-a}\neq0$, these $a$ are congruent modulo $N$, and $a=0$ occurs since $u\nmid F$. If only $a=0$ occurs, $C$ is a cone. This leaves the three cases ($N\ge5$ is impossible), and in case $N=2$ the coefficient $c$ of $u^4$ is non-zero because otherwise $O$ is a
singular point. At $P=(0:v_0:w_0)\in C\cap\ell$ we have $\partial_uF(P)=0$, so $C$ is smooth at $P$ iff $(v_0:w_0)$ is a simple root of $F_4$, and then
$T_P=\{\partial_vF_4(P)v+\partial_wF_4(P)w=0\}$ passes through $O$.
\end{proof}

\begin{theorem}\label{thm:homologies}
Let $\sigma$ be a homology of order $N$ preserving a smooth quartic $C$, and let $P$ be one of the four points of $C$ on its axis.
\begin{enumerate}[label=\rm(\alph*)]
\item $N=2$: $P$ is either a hyperflex or a non-flex sextactic point with $w_2(P)\in\{1,3\}$; in the latter case the vanishing sequence of $V_3$ consists of six even and four odd numbers, $w_3(P)\in\{1,3\}$, and $w_3(P)=1$ iff the sequence is $\{0,\ldots,8,10\}$.
\item $N=3$: $P$ is an ordinary flex with vanishing sequence $\{0,\ldots,7,9,12\}$ of $V_3$, so $w_3(P)=4$. Explicitly, with $T_P=\V(t)$ and $(v_1:w_1)$ the root of $F_1$, put $\lambda=F_4(v_1,w_1)/t(v_1,w_1)^4$ and $B=(\lambda t^4-F_4)/F_1$; then $E_P:u^3=B(v,w)$ satisfies
  \[ F=F_1\cdot(u^3-B)+\lambda\,t^4,\qquad E_P\cdot T_P=3P,\qquad E_P\cdot C=12P. \]
Moreover the centre $O$ is a hyperflex with tangent line $\V(F_1)$.
\item $N=4$: $P$ is a hyperflex; the four hyperflexes on the axis are collinear and their tangent lines pass through $O$.
\end{enumerate}
\end{theorem}

\begin{proof}
By Proposition~\ref{prop:homnormal}, $\ell$ is transversal to $C$ at $P$, so $\sigma$ acts trivially on $T_P\ell$ and, being non-trivial on $C$ near its isolated fixed point $P$, by a primitive $N$-th root of unity on $T_PC$. Hence $\tau=u/G_0$, with $G_0$ a form in $v,w$ not vanishing at $P$, is a local coordinate at $P$ with $\sigma^*\tau=\zeta\tau$. The space $S_n$ splits into the eigenspaces $S_n^{(a)}$ spanned by the monomials $u^iv^jw^k$ with $i\equiv a\pmod N$, and for $G\in S_n^{(a)}$ the function $G/G_0^{n}$ is a $\zeta^a$-eigenfunction, so its expansion in $\tau$
contains only powers $\tau^k$ with $k\equiv a\pmod N$. Since sections with orders in different residue classes cannot cancel, the vanishing sequence of $V_n$ is the union over $a$ of $\dim S_n^{(a)}$ distinct numbers $\equiv a\pmod N$ (for $n\le 3$).

(a) For $V_1$: two even orders from $\{0,2,4\}$ ($v,w$) and the odd order $1$ ($u$), so the sequence is $\{0,1,2\}$ or $\{0,1,4\}$: $P$ is not an ordinary flex. If $P$ is not a flex, $V_2$ has four even and two odd orders (from $v^2,vw,w^2,u^2$ and $uv,uw$), so $m$ in Lemma~\ref{lem:nonflex} is even, $m\in\{6,8\}$. For $V_3$ the dimensions are $6$ ($v^3,\dots,w^3,u^2v,u^2w$) and $4$ ($u^3,uv^2,uvw,uw^2$); six distinct even and four distinct odd numbers sum to at least $30+16=46>45$, with equality only for $\{0,2,4,6,8,10\}\cup\{1,3,5,7\}$. By Theorem~\ref{thm:KE} a $1$-sextactic point has $w_3\in\{0,1,3\}$ and a $3$-sextactic one $w_3=3$; parity excludes $0$.

(b) For $V_1$: residue $0$ gives two orders from $\{0,3\}$, residue $1$ gives the order $1$ of $u$; so the sequence is $\{0,1,3\}$ and $P$ is an ordinary flex. For $V_3$: $\dim S_3^{(0)}=5$ ($v^3,v^2w,vw^2,w^3,u^3$), so all of $0,3,6,9,12$ occur; residues $1$ and $2$ contribute three and two numbers, which contain $\{1,4,7\}$ and $\{2,5\}$ by
Lemma~\ref{lem:flex}. Hence the sequence is $\{0,\ldots,7,9,12\}$. The explicit formula: $\lambda t^4-F_4$ vanishes at the root of $F_1$, so $B$ is a binary cubic, and $B(P)=0$ because $t(P)=F_4(P)=0$. On $T_P$, i.e.\ on the points $(u:sv_0:sw_0)$, the
cubic $u^3-B$ restricts to $u^3$, so $E_P\cdot T_P=3P$, and $E_P\cdot C=E_P\cdot(\lambda t^4)=12P$. Finally $O\in C$ since $F$ has no $u^4$ term, and on the line $\V(F_1)$ through $O$ the form $F$ restricts to $F_4(v_1,w_1)s^4$ with $F_4(v_1,w_1)\neq0$; so $\V(F_1)$ meets $C$ only at $O$, with multiplicity $4$.

(c) For $V_1$: residue $0$ gives two orders from $\{0,4\}$, residue $1$ the order $1$; the sequence $\{0,1,4\}$ means $I_P(C,T_P)=4$.
\end{proof}

\begin{corollary}\label{cor:axes}
Every point of $C$ on the axis of an involution in $\Aut(C)$ is a $3$-Weierstrass point, and is either a hyperflex or a sextactic point. Every point on the axis of a homology of order $3$ (resp.\ $4$) is an ordinary flex of $3$-weight $4$ (resp.\ a hyperflex).
\end{corollary}

The parity argument works in every order, which gives a partial answer to the question behind the conjecture of \cite{KamelElshareef}, see Section~\ref{sec:fermat4}.

\begin{theorem}\label{thm:allq}
Let $C$ be a smooth plane quartic and $q\ge2$. Then $\W_q(C)$ contains all hyperflexes of $C$, all points on axes of involutions preserving $C$, and all points on axes of homologies of order three preserving $C$.
\end{theorem}

\begin{proof}
Hyperflexes: Lemma~\ref{lem:hyperflex}. For the other points let $\sigma$ be the homology, $N\in\{2,3\}$, and let $h_a(q)$ be the dimension of the
$\zeta^a$-eigenspace of $\sigma$ on $V_q=S_q/F\cdot S_{q-4}$ (note $F\in S^{(0)}$). If $w_q(P)=0$, the vanishing sequence is $0,1,\dots,N_q-1$, and by the proof of Theorem~\ref{thm:homologies} $h_a(q)$ equals $b_a(q)=\#\{0\le i<N_q:\ i\equiv a\pmod
N\}$. But a count of monomials gives
\[ (h_0,h_1)=(b_0+1,\,b_1-1)\quad(N=2,\ q\ge2),\qquad
  (h_0,h_1,h_2)=(b_0+1,\,b_1,\,b_2-1)\quad(N=3,\ q\ge1).\]
Indeed, regard $h(q)=\sum_ah_a(q)\chi^a$ as an element of the representation ring of $\ZZ/N$, $\chi$ the character of $u$. Then $\sum_qh(q)t^q=(1-t^4)/\bigl((1-\chi t)(1-t)^2\bigr)$, hence
\[ (1-t^N)\sum_qh(q)t^q=\frac{(1+t+t^2+t^3)(1+\chi t+\dots+\chi^{N-1}t^{N-1})}{1-t}, \]
so $h(q)-h(q-N)=4(1+\chi+\dots+\chi^{N-1})$ for $q\ge N+2$. The same holds for $b(q)=\sum_ab_a(q)\chi^a$ when $q-N\ge2$, since then $[N_{q-N},N_q)$ consists of $4N$ consecutive integers. The identities are checked directly for $q\le N+1$ (and, in \sing{general.sing}, for all $q\le60$).
\end{proof}

In particular, the fixed points of involutions are $q$-Weierstrass points for all $q\ge2$, whereas Lemma~\ref{lem:flex} and Theorem~\ref{thm:KE} show that an ordinary flex need not be a $3$-Weierstrass point. For a general quartic no flex is, see Remark~\ref{rem:general} below.

\begin{remark}\label{rem:general}
For a general smooth quartic none of the special behaviour above occurs. Consider
\begin{align*}
F_0={}&-5X^4+5X^3Z-6X^2Y^2-4X^2YZ+6X^2Z^2+7XY^3-4XY^2Z\\
&-3XYZ^2+9XZ^3+7Y^4+5Y^3Z+5Y^2Z^2+8YZ^3+5Z^4.
\end{align*}
The script \sing{generic.sing} computes $\Omega_1,\Omega_2,\Omega_3$ of $C_0=\V(F_0)$
over $\QQ$, checks over $\QQ$ that $\W_1(C_0)$ and $\W_2(C_0)$ are reduced ($24$ and $108$ distinct points), and checks modulo the primes $32003$ and $1000003$ that $\V(F_0,\Omega_3)$ consists of $300$ distinct points and that $\V(F_0,\Omega_3,\Hess)=\V(F_0,\Omega_3,\Omega_2)=\emptyset$. Since $\V(F_0,\Omega_3)$ is finite and flat of degree $300$ over $\ZZ_{(p)}$, étaleness of the special fibre implies that of the generic fibre, and empty special fibres force empty generic fibres. Hence $\W_1(C_0)$, $\W_2(C_0)$ and $\W_3(C_0)$ are reduced, consisting of $24$, $108$ and $300$ distinct points, and no point of $\W_3(C_0)$ is a
flex or a sextactic point. These properties are open in the space of smooth quartics, so for a \emph{general} smooth quartic the $3$-Weierstrass points are $300$ distinct dectactic points. In particular, the case $m=8$ of Lemma~\ref{lem:flex} is the general one, and sextactic points are in general not $3$-Weierstrass points.
\end{remark}

\section{The Fermat quartic}\label{sec:fermat}

Let $\Fermat:\ X^4+Y^4+Z^4=0$. Its automorphism group has order $96$, see \cite{KuribayashiKomiya,Dolgachev}. This group is generated by the diagonal matrices with entries
in $\{1,i\}$ and the permutation matrices, and it contains exactly $15$ involutions, three cyclic groups of homologies of order four and no homology of order three (\sing{groups.sing}). These homologies are: \begin{itemize}[leftmargin=2em]
\item the three homologies of order four $(X:Y:Z)\mapsto(iX:Y:Z)$, $(X:iY:Z)$, $(X:Y:iZ)$, with axes $X=0$, $Y=0$, $Z=0$; their squares are the sign changes, involutions with the same axes;
\item the twelve involutions such as $(X:Y:Z)\mapsto(\varepsilon Y:\varepsilon^{-1}X:Z)$,
  $\varepsilon^4=1$, with axes $X=\varepsilon Y$, $Y=\varepsilon Z$, $Z=\varepsilon X$.
\end{itemize}

\paragraph{Flexes.} By Theorem~\ref{thm:homologies}(c) the twelve points $\Fermat\cap\V(XYZ)$, namely $(0:\zeta:1)$, $(\zeta:0:1)$, $(\zeta:1:0)$ with
$\zeta^4=-1$, are hyperflexes. Directly, the line $Y=\zeta Z$ meets $\Fermat$ only at $(0:\zeta:1)$, with multiplicity $4$. Since $2\cdot 12=24=\deg\W_1(\Fermat)$, there are no other flexes (indeed $\Hess(F)=1728\,X^2Y^2Z^2$). Let $H$ denote the divisor of these twelve points.

\paragraph{The axis points.} Each of the twelve lines $X=\varepsilon Y$, $Y=\varepsilon Z$, $Z=\varepsilon X$ meets $\Fermat$ in four distinct points (e.g.\ $X=Y$ gives $2Y^4+Z^4=0$), none of them a hyperflex, and two such lines have no common point on $\Fermat$. Let $S$ be the resulting divisor of $48$ points; it is cut out transversally by
\[ \Delta_{12}=(X^4-Y^4)(Y^4-Z^4)(Z^4-X^4). \]
By Theorem~\ref{thm:homologies}(a) the points of $S$ are sextactic and $3$-Weierstrass.

\begin{theorem}[Fermat quartic]\label{thm:fermat}
Put $e_2=X^4Y^4+Y^4Z^4+Z^4X^4$, $e_3=X^4Y^4Z^4$ and
\[ \Pi_{48}=144\,e_2^{6}+9624\,e_2^{3}e_3^{2}-22379\,e_3^{4}. \]
Then, modulo $X^4+Y^4+Z^4$ and up to non-zero scalars,
\[ \Omega_2\equiv X^5Y^5Z^5\,\Delta_{12},\qquad
  \Omega_3\equiv X^5Y^5Z^5\,\Delta_{12}\,\Pi_{48}.
\]
Consequently:
\begin{enumerate}[label=\rm(\alph*)]
\item $\W_2(\Fermat)=5H+S$: the sextactic points are the $48$ axis points, each of $2$-weight $1$.
\item $\W_3(\Fermat)=5H+S+D$, where $D=\Fermat\cap\V(\Pi_{48})$ consists of $192$
  distinct dectactic points. Thus
\end{enumerate}
\[
\begin{array}{lccc}
\toprule
\textup{type} & \textup{number} & \textup{vanishing sequence of }V_3 & 3\textup{-weight}\\
\midrule
\textup{hyperflexes} & 12 & 0,1,2,3,4,5,6,8,9,12 & 5\\
\textup{sextactic points} & 48 & 0,1,\ldots,8,10 & 1\\
\textup{dectactic points} & 192 & 0,1,\ldots,8,10 & 1\\
\bottomrule
\end{array}
\]
and $\deg\W_3(\Fermat)=60+48+192=300$, with $252$ distinct $3$-Weierstrass points.
\end{theorem}

\begin{proof}
On $\Fermat$ the divisor of $X^5Y^5Z^5$ is $5H$ (each coordinate line meets $\Fermat$ transversally in four hyperflexes) and that of $\Delta_{12}$ is $S$. By Lemma~\ref{lem:hyperflex} and Theorem~\ref{thm:homologies}, $\W_n(\Fermat)\ge 5H+S$ for $n=2,3$.

(a) Both sides have degree $108$, so $\W_2(\Fermat)=5H+S=\divisor(X^5Y^5Z^5\Delta_{12})$,
and two forms of degree $27$ with the same divisor on $\Fermat$ agree modulo $F$ up to a scalar. The points of $S$ are not flexes, hence sextactic.

(b) The divisor $\W_3(\Fermat)-5H-S$ is effective of degree $192$ and lies in $|\OO_{\Fermat}(75-27)|$, so it is cut out by a form $\Pi$ of degree $48$ with
$\Omega_3\equiv X^5Y^5Z^5\Delta_{12}\Pi$. That $\Pi\equiv\Pi_{48}$ is verified in \sing{fermat.sing} with the Hamiltonian formula. In the affine chart it reads
\begin{equation}\label{eq:fermat-affine}
  W^{(3)}_x=c\,\frac{x^5P_{12}(x)P_{48}(x)}{y^{130}},\qquad
  P_{12}(x)=(x^4-1)(x^4+2)(2x^4+1),
\end{equation}
\[
\begin{aligned}
P_{48}(x)={}&144x^{48}+864x^{44}-6600x^{40}-40920x^{36}-134531x^{32}-283100x^{28}-364194x^{24}\\
&-283100x^{20}-134531x^{16}-40920x^{12}-6600x^{8}+864x^{4}+144,
\end{aligned}
\]
in accordance with \cite[Section~4.1]{KamelElshareef}, where we used $f_y^{45}=(4y^3)^{45}$,
$\Delta_{12}(x,y,1)\equiv P_{12}(x)$ and $\Pi_{48}(x,y,1)\equiv P_{48}(x)$ on $\Fermat$.
As a polynomial in $u=x^4$, $P_{48}$ is irreducible over $\QQ$, square-free, and its roots are different from $0,-1,1,-2,-\tfrac12$. Hence each of its $48$ roots $x$ gives four distinct points of $\Fermat$ (as $y^4=-1-x^4\neq0$), and these $192$ points are
distinct from $H$ and $S$. There are no points of $D$ at infinity, since there $e_3=0$ and $e_2=-X^8\neq 0$. The degree count forces weight $1$ at every point of $D$, and the points of $D$ are neither flexes nor sextactic by (a). The vanishing sequences follow from Lemma~\ref{lem:hyperflex} and from the fact that $\{0,\ldots,8,10\}$ is the only
vanishing sequence of weight $1$.
\end{proof}

\begin{remark}
The polynomial $\Pi_{48}$ is symmetric in $X^4,Y^4,Z^4$, whereas $\Delta_{12}\Pi_{48}$
is alternating under permutations of these variables, in accordance with Corollary~\ref{cor:covariance}: a transposition $g$ of two coordinates has $\det g=-1$ and $\Omega_3\circ g\equiv-\Omega_3$. In terms of $a=X^4$, $c=Z^4$ and $b=-a-c$ the identity $\Pi_{48}(a,b,c)=\sum_k p_k\,a^{12-k}c^k$, where $p_k$ are the
coefficients of $P_{48}$, can be checked by hand. For the second-order divisor of the Kuribayashi family, to which $\Fermat$ belongs, see \cite{AlwaleedSakai}.
\end{remark}

\subsection{Order four and a conjecture of Kamel and Elshareef}\label{sec:fermat4}

Kamel and Elshareef observed that $\W_1(\Fermat)\subseteq\W_2(\Fermat)\subseteq
\W_3(\Fermat)$ and conjectured \cite[Section~5]{KamelElshareef} that $\W_q(\Fermat)\subseteq\W_{q+1}(\Fermat)$ for every $q\ge1$. The first two inclusions are explained, and extended, by symmetry: by Theorem~\ref{thm:allq} and Theorem~\ref{thm:fermat}(a),
\begin{equation}\label{eq:W12inWq}
\operatorname{supp}\W_1(\Fermat)\cup\operatorname{supp}\W_2(\Fermat)=H\cup S\subseteq\operatorname{supp}\W_q(\Fermat)\qquad\text{for all }q\ge2,
\end{equation}
because every point of $H\cup S$ is a hyperflex or lies on the axis of an involution. The dectactic points have no such reason to persist, and in fact they do not.

\begin{theorem}\label{thm:fermat4}
Modulo $X^4+Y^4+Z^4$ and up to a non-zero scalar,
$\Omega_4\equiv X^5Y^5Z^5\,\Delta_{12}\,R_{120}$, where $R_{120}$ is a form of degree $120$ such that $\Fermat\cap\V(R_{120})$ consists of $480$ distinct points, disjoint from $H$, $S$ and $D$. Hence
\[ \W_4(\Fermat)=5H+S+R,\qquad 60+48+480=588, \]
with $540$ distinct $4$-Weierstrass points, and $D\cap\operatorname{supp}\W_4(\Fermat)=\emptyset$.
In particular, $\W_3(\Fermat)\not\subseteq\W_4(\Fermat)$, and the conjecture of \cite{KamelElshareef} fails for $q=3$.
\end{theorem}

\begin{proof}
The script \sing{fermat.sing} computes $\Omega_4$ by Proposition~\ref{prop:hamilton} from the $14$ standard monomials of degree $4$, checks that $\Omega_4\in(F,X^5Y^5Z^5\Delta_{12})$, computes $R_{120}$ by division modulo $F$, and checks that $\Omega_4\equiv X^5Y^5Z^5\Delta_{12}R_{120}$, that $\V(F,R_{120})$ consists of $480$ distinct points and that it is disjoint from $\V(F,XYZ\,\Delta_{12}\Pi_{48})$, all over $\QQ$. Since $\divisor(X^5Y^5Z^5\Delta_{12})=5H+S$ on $\Fermat$, this proves the claims.
\end{proof}

As an independent check, which does not use $\Omega_4$, \sing{fermat.sing} also computes the vanishing sequences at the point $(3105:678:1)$ of $\V(F,\Pi_{48})$ over $\FF_{33037}$, where the two curves meet transversally so that, by Hensel's lemma,
the point lifts to a point $\tilde P$ of $D$ over $\QQ_{33037}$. They are $\{0,\ldots,8,10\}$ for $V_3$ and $\{0,\ldots,13\}$ for $V_4$. Vanishing orders can only increase under reduction, so $w_4(\tilde P)=0$.

\begin{remark}
The same happens on the other quartics of this paper: on $\KK$, $\Klein$ and $\Pic$ no dectactic point is a $4$-Weierstrass point (\sing{kk.sing}, \sing{klein.sing}, \sing{picard.sing}), and on $\Klein$ not even the flexes are, see Proposition~\ref{prop:klein4}. On the other hand, all points in Theorem~\ref{thm:allq} persist, in accordance with the theorem.
\end{remark}

\section{The Komiya--Kuribayashi quartic}\label{sec:kk}

Let
\[ \KK:\ X^4+Y^4+Z^4+3(X^2Y^2+Y^2Z^2+Z^2X^2)=0,\]
be the member $a=3$ of the Ciani's pencil of quartics
$C_a:X^4+Y^4+Z^4+a(X^2Y^2+Y^2Z^2+Z^2X^2)=0$ \cite{KuribayashiKomiya,AlwaleedSakai}.
The $24$ projective transformations given by signed permutation matrices form a group $G_{24}\cong S_4$ preserving every $C_a$. Its nine involutions are harmonic homologies with axes
\[X=0,\quad Y=0,\quad Z=0,\qquad X=\pm Y,\quad Y=\pm Z,\quad Z=\pm X.\]
Its elements of order $3$ and $4$ have three distinct eigenvalues (e.g.\
$(X:Y:Z)\mapsto(Y:Z:X)$ and $(X:Y:Z)\mapsto(Y:-X:Z)$), so they are not homologies; the nine involutions are the only homologies in $G_{24}$ (\sing{groups.sing}).

\begin{proposition}\label{prop:kk-flexes}
\begin{enumerate}[label=\rm(\roman*)]
\item $\KK$ is smooth.
\item $\KK$ has exactly twelve flexes, all hyperflexes:
  \[
    (\pm1:\pm i:1)\quad\textup{and}\quad(\pm i:\varepsilon:1),\ \varepsilon\in\{\pm1,\pm i\}.
  \]
\item Each of the six lines $X=\pm iY$, $Y=\pm iZ$, $Z=\pm iX$ contains four hyperflexes, and each hyperflex lies on exactly two of them. Consequently $\KK$ is not projectively equivalent to the Fermat quartic, and $\Omega_1\equiv c\,L_+$ with
 $L_+=(X^2+Y^2)(Y^2+Z^2)(Z^2+X^2)$.
\end{enumerate}
\end{proposition}

\begin{proof}
(i) The Jacobian ideal $(F_X,F_Y,F_Z)$ contains $(X,Y,Z)^7$, so the partial derivatives have no common zero.

(ii) The Hessian of $\KK$ is $216\,h$ with
$h=2\sum X^6+3\sum_{\mathrm{sym}}X^4Y^2+8X^2Y^2Z^2$, and
\[ \operatorname{Res}_y\bigl(f,h(x,y,1)\bigr)=c\,(x-1)^4(x+1)^4(x^2+1)^8,\]
while $\KK$ and $\V(h)$ do not meet on $Z=0$. Solving gives the twelve listed points, each with intersection multiplicity $2$. They are hyperflexes: at $(1:i:1)$ the tangent line is $X+2iY+Z=0$, and on it $F=5(Y-iZ)^4$; the others follow by symmetry. As
$2\cdot12=24$, there are no other flexes, see also \cite{Edge}.

(iii) On $X=iY$ one has $F=Z^4-X^4$, giving $(1:-i:\varepsilon)$, $\varepsilon^4=1$, which are four hyperflexes, and the other lines follow by symmetry. On the Fermat quartic the twelve hyperflexes lie by fours on the three coordinate lines, and any other line contains at most one hyperflex from each coordinate line, hence at most three. Since projective equivalences preserve hyperflexes and collinearity, the claim follows. Finally $\divisor(L_+|_{\KK})=2H=\W_1(\KK)$, so $\Omega_1\equiv cL_+$.
\end{proof}

So on $\KK$, unlike on $\Fermat$, the hyperflexes are not forced by homologies: each
lies on the axis of an involution (see $\divisor(L_-)$ below), but by
Theorem~\ref{thm:homologies}(a) this only forces a hyperflex or a sextactic point. Write $H$ for the divisor of the twelve hyperflexes and put
\[ L_-=(X^2-Y^2)(Y^2-Z^2)(Z^2-X^2),\qquad
  S_6=45\,(X^2+Y^2+Z^2)^3-11\,X^2Y^2Z^2 .\]
On $\KK$:
\begin{itemize}[leftmargin=2em]
\item $\divisor(L_+)=2H$, by Proposition~\ref{prop:kk-flexes}(iii);
\item $\divisor(XYZ)=A$, the $12$ points on the coordinate axes; e.g.\ $X=0$ gives
  $Y^4+3Y^2Z^2+Z^4=0$;
\item $\divisor(L_-)=H+B$: on $X=Y$ one has $F=(5X^2+Z^2)(X^2+Z^2)$, so each of the
  six lines $X=\pm Y,\dots$ contains two hyperflexes and two further points such as  $(1:1:\pm i\sqrt5)$; $B$ is the divisor of these $12$ further points;
\item $\divisor(S_6)=S'$, a free $G_{24}$-orbit of $24$ points.
\end{itemize}
The points of $A\cup B$ are exactly the non-hyperflex points on the axes of the nine involutions. The sextic $S_6$ is congruent modulo $F$ to
$Q_6=90X^2Y^4+124X^2Y^2Z^2+90X^2Z^4+45Y^6+135Y^4Z^2+135Y^2Z^4+45Z^6$, and it does not anish at $H\cup A\cup B$.

\begin{theorem}[Komiya--Kuribayashi quartic]\label{thm:kk}
Set $U=X^2+Y^2+Z^2$ and $V=X^2Y^2Z^2$, and let
\[
\begin{aligned}
P_9(t)={}&99421875-682428125t-2040212500t^2
   +16403537500t^3\\
 &-83782723750t^4+94921278250t^5+4794299900t^6\\
 &+18371470540t^7-2821591285t^8-163782773t^9.
\end{aligned}
\]
With $c_j=[t^j]P_9(t)$, define the $G_{24}$-invariant form
\[
  \Psi_{54}=\sum_{j=0}^9 c_j U^{27-3j}V^j.
\]
Modulo $F$ and up to non-zero scalars,
\[
  \Omega_2\equiv L_+^2\,L_-\,XYZ\,S_6,\qquad
  \Omega_3\equiv L_+^2\,L_-\,XYZ\,\Psi_{54}.
\]
More precisely:
\begin{enumerate}[label=\rm(\alph*)]
\item $\W_2(\KK)=5H+A+B+S'$; the $48$ sextactic points are $A\cup B\cup S'$, each of $2$-weight $1$.
\item $\W_3(\KK)=5H+A+B+D$, where $D=\KK\cap\V(\Psi_{54})$ consists of $216$ distinct dectactic points of weight $1$, forming nine free $G_{24}$-orbits of $24$ points each. In particular, the $24$ sextactic points of $S'$ are \emph{not} $3$-Weierstrass points. Thus
\end{enumerate}
\[
\begin{array}{lccc}
\toprule
\textup{type} & \textup{number} & \textup{vanishing sequence of }V_3 & 3\textup{-weight}\\
\midrule
\textup{hyperflexes} & 12 & 0,1,2,3,4,5,6,8,9,12 & 5\\
\textup{sextactic points on axes }(A\cup B) & 24 & 0,1,\ldots,8,10 & 1\\
\textup{sextactic points on }S_6=0 & 24 & 0,1,\ldots,9 & 0\\
\textup{dectactic points} & 216 & 0,1,\ldots,8,10 & 1\\
\bottomrule
\end{array}
\]
and $\deg\W_3(\KK)=60+24+216=300$, with $252$ distinct $3$-Weierstrass points.
\end{theorem}

\begin{proof}
On $\KK$, $\divisor(L_+^2L_-XYZ)=4H+(H+B)+A=5H+A+B$, and by Lemma~\ref{lem:hyperflex}
and Theorem~\ref{thm:homologies} we have $\W_n(\KK)\ge 5H+A+B$ for $n=2,3$. The residual divisors $\W_2(\KK)-5H-A-B$ and $\W_3(\KK)-5H-A-B$ are effective, of degrees $24$ and $216$, and lie in $|\OO_{\KK}(6)|$ and $|\OO_{\KK}(54)|$. Hence they are cut out by forms of degrees $6$ and $54$, which proves the shape of $\Omega_2$ and $\Omega_3$.
We take $G_{24}$ to consist of determinant-one signed permutation matrices, so Corollary~\ref{cor:covariance} implies that $\Omega_3$ is $G_{24}$-invariant modulo $F$. The product $L_+^2L_-XYZ$ is invariant as well: under a signed permutation of determinant one, the signs acquired by $L_-$ and by $XYZ$ cancel. Thus the residual class $\Psi_{54}$ is invariant modulo $F$.
Averaging over $G_{24}$ gives an invariant polynomial representative of this class. In even degree every invariant form is symmetric in $X^2,Y^2,Z^2$: the even sign changes force each exponent to be even, and the group acts as the full symmetric group on the squared coordinates. Writing
$T=X^2Y^2+Y^2Z^2+Z^2X^2$, we have $F=U^2+T$. Hence every invariant class of degree $54$ admits a representative $\sum_{j=0}^{9} c_j U^{27-3j}V^j$, as asserted.

For (a), no point of $\KK\cap\V(S_6)$ has a non trivial stabiliser in $G_{24}$: every non-trivial subgroup of $S_4$ contains an element of order $2$ or $3$; a point of $\KK$ fixed by an involution lies on its axis (the centres are not on $\KK$, by
Proposition~\ref{prop:homnormal}), where $S_6\neq0$; and the fixed points on $\KK$ of the elements of order $3$, such as $(1:\omega:\omega^2)$, satisfy $X^2+Y^2+Z^2=0$, so that $S_6=-11X^2Y^2Z^2\neq0$ there. Hence $\KK\cap\V(S_6)$, which has degree $24$, is a single free orbit $S'$ of $24$ distinct points. The identity $\Omega_2\equiv cL_+^2L_-XYZS_6$ is verified in \sing{kk.sing}.

For (b), an exact calculation using Proposition~\ref{prop:hamilton}, with the ordering of cubic monomials fixed in Section~\ref{sec:wronskian}, gives
\[
  \det(\delta^i b_j)_{0\le i,j\le9}
  \equiv -2^{63}3^{14}5^{15}7^3\,
  L_+^2L_-XYZ\,\Psi_{54}\pmod F.
\]
The script \sing{kk.sing} obtains $\Psi_{54}$ directly, as the quotient of $\Omega_3$ by $L_+^2L_-XYZ$ modulo $F$, and checks the conclusions of the following argument (that $\V(F,\Psi_{54})$ consists of $216$ distinct points disjoint from $H\cup A\cup B\cup S'$, and the $G_{24}$-invariance of $\Psi_{54}$) without using the explicit form; the explicit form can be confirmed by comparing it with this quotient.

The morphism
\[ \rho_{\KK}=(V:U^3):\KK\longrightarrow\PP^1 \]
is base-point-free (the equations $U=V=0$ have no solution on $\KK$) and has degree $24$. It is $G_{24}$-invariant, hence it is the quotient map by $G_{24}$, and its fibres are precisely the $G_{24}$-orbits. Its branch values are $0,-1,5/27,\infty$, corresponding to the orbits $A,H,B$ and the orbit of points fixed by elements of order three, respectively. These exhaust the ramification by Riemann--Hurwitz, since three orbits of size $12$ with ramification index $2$ and one orbit of size $8$
with ramification index $3$ contribute $3\cdot12+8\cdot2=52=2g(\KK)-2+48$. The polynomial $P_9$ is square-free and is non-zero at all four branch values (including $\infty$, since its leading coefficient is non-zero). Moreover $P_9(45/11)\neq0$, and $45/11$ is the value of $\rho_{\KK}$ on $S'$, since $S_6=45U^3-11V$.
Thus $D=\KK\cap\V(\Psi_{54})$ is the union of nine distinct free $G_{24}$-orbits, each of $24$ points, disjoint from $H\cup A\cup B\cup S'$. All points of $D$ are therefore neither flexes nor sextactic points, and they have $3$-weight one.
\end{proof}

\begin{remark}
The numerical shape $12\cdot 5+240\cdot 1$ of $\W_3$ is the same for $\Fermat$ and $\KK$, but the decomposition is not: on $\Fermat$ every sextactic point is a $3$-Weierstrass point, because all of them lie on axes of involutions, whereas on $\KK$ only the $24$ sextactic points on axes are. This gives an explicit symmetric example of sextactic points that are not $3$-Weierstrass points (compare
Remark~\ref{rem:general} and the Picard quartic in Section~\ref{sec:picard}). We have not expanded $\Psi_{54}$ in monomials; it is produced by \sing{kk.sing}. For maximal tangency lines on quartics see also \cite{MertaZielinski}.
\end{remark}

\section{The Klein quartic}\label{sec:klein}

\subsection{Klein's invariants}

Let
\[\Klein=\V(\Phi_4)\subset\PP^2,\qquad \Phi_4=X^3Y+Y^3Z+Z^3X .\]
Its automorphism group $G=\Aut(\Klein)\cong\PSL_2(\FF_7)$ has order $168$ and is the image of a subgroup of $\SL_3(\CC)$ isomorphic to $G$ which fixes $\Phi_4$, so $G$ acts
linearly on $\CC[X,Y,Z]$, see \cite{Klein,Elkies}. Following Klein, put
\[ \Phi_6=-\tfrac{1}{54}\Hess(\Phi_4)=X^5Z+XY^5+YZ^5-5X^2Y^2Z^2, \]
\[  \Phi_{14}=\frac{1}{9}\det
  \begin{pmatrix}
  \Phi_{4,XX}&\Phi_{4,XY}&\Phi_{4,XZ}&\Phi_{6,X}\\
  \Phi_{4,YX}&\Phi_{4,YY}&\Phi_{4,YZ}&\Phi_{6,Y}\\
  \Phi_{4,ZX}&\Phi_{4,ZY}&\Phi_{4,ZZ}&\Phi_{6,Z}\\
  \Phi_{6,X}&\Phi_{6,Y}&\Phi_{6,Z}&0
  \end{pmatrix},
  \qquad
  \Phi_{21}=\frac{1}{14}\det
  \begin{pmatrix}
  \Phi_{4,X}&\Phi_{4,Y}&\Phi_{4,Z}\\
  \Phi_{6,X}&\Phi_{6,Y}&\Phi_{6,Z}\\
  \Phi_{14,X}&\Phi_{14,Y}&\Phi_{14,Z}
  \end{pmatrix}. \]
Explicitly,
\begin{align*}
\Phi_{14}={}&X^{14}-34X^{11}Y^2Z 250X^9YZ^4+375X^8Y^4Z^2+18X^7Y^7+18X^7Z^7
-126X^6Y^3Z^5\\
&-126X^5Y^6Z^3-250X^4Y^9Z+375X^4Y^2Z^8-126X^3Y^5Z^6+375X^2Y^8Z^4\\
&-34X^2YZ^{11}-34XY^{11}Z^2-250XY^4Z^9+Y^{14}+18Y^7Z^7+Z^{14},
\end{align*}
and $\Phi_{21}$ factors over $\QQ$ as $\Phi_{21}=G_3\,G_6^{(1)}G_6^{(2)}G_6^{(3)}$ with
\begin{align*}
G_3={}&X^3-X^2Y-2X^2Z-2XY^2+6XYZ-XZ^2+Y^3-Y^2Z-2YZ^2+Z^3,\\[4pt]
G_6^{(1)}={}&X^6-2X^5Y-4X^5Z+4X^4Y^2+2X^4YZ+9X^4Z^2-8X^3Y^3-6X^3Y^2Z+2X^3YZ^2\\
&-8X^3Z^3+9X^2Y^4+2X^2Y^3Z+13X^2Y^2Z^2-6X^2YZ^3+4X^2Z^4-4XY^5+2XY^4Z\\
&-6XY^3Z^2+2XY^2Z^3+2XYZ^4-2XZ^5+Y^6-2Y^5Z+4Y^4Z^2-8Y^3Z^3+9Y^2Z^4\\
&-4YZ^5+Z^6,\\[4pt]
G_6^{(2)}={}&X^6-2X^5Y+3X^5Z+4X^4Y^2-5X^4YZ+2X^4Z^2-X^3Y^3+8X^3Y^2Z+9X^3YZ^2\\
&-X^3Z^3+2X^2Y^4+9X^2Y^3Z-X^2Y^2Z^2+8X^2YZ^3+4X^2Z^4+3XY^5-5XY^4Z\\
&+8XY^3Z^2+9XY^2Z^3-5XYZ^4-2XZ^5+Y^6-2Y^5Z+4Y^4Z^2-Y^3Z^3+2Y^2Z^4\\
&+3YZ^5+Z^6,\\[4pt]
G_6^{(3)}={}&X^6+5X^5Y+3X^5Z+11X^4Y^2+16X^4YZ+9X^4Z^2+13X^3Y^3+36X^3Y^2Z\\
&+37X^3YZ^2+13X^3Z^3+9X^2Y^4+37X^2Y^3Z+55X^2Y^2Z^2+36X^2YZ^3+11X^2Z^4\\
&+3XY^5+16XY^4Z+36XY^3Z^2+37XY^2Z^3+16XYZ^4+5XZ^5+Y^6+5Y^5Z\\
&+11Y^4Z^2+13Y^3Z^3+9Y^2Z^4+3YZ^5+Z^6.
\end{align*}
Geometrically, $\Phi_{21}=0$ is the union of the $21$ axes of the involutions of $G$. These identities, the irreducibility of the four factors over $\QQ$, and the relation
\begin{equation}\label{eq:klein-relation}
  \Phi_{21}^2\equiv\Phi_{14}^3-1728\,\Phi_6^7\pmod{\Phi_4},
\end{equation}
are checked in \sing{klein.sing}. We use the following classical facts \cite{Klein,Elkies,Dolgachev}; the transversality statements in \ref{K:orbits} are also checked in \sing{klein.sing}.
\begin{enumerate}[label=(K\arabic*)]
\item\label{K:ring} $\CC[X,Y,Z]^G=\CC[\Phi_4,\Phi_6,\Phi_{14}]\oplus
  \Phi_{21}\,\CC[\Phi_4,\Phi_6,\Phi_{14}]$;
  equivalently its Molien series is
  $(1+t^{21})/\bigl((1-t^4)(1-t^6)(1-t^{14})\bigr)$.
\item\label{K:orbits} The $G$-orbits on $\Klein$ with non-trivial stabiliser are $O_{24}$ (stabiliser of order $7$: the $24$ flexes), $O_{56}$ (order $3$: the contact points of the $28$ bitangents) and $O_{84}$ (order $2$: the fixed points of the $21$ involutions); they are cut out transversally by $\Phi_6$, $\Phi_{14}$ and $\Phi_{21}$,
  i.e.\ $\divisor(\Phi_6|_\Klein)=O_{24}$, $\divisor(\Phi_{14}|_\Klein)=O_{56}$ and
  $\divisor(\Phi_{21}|_\Klein)=O_{84}$.
\end{enumerate}

\begin{lemma}\label{lem:quotient}
The map $\pi (\Phi_{14}^3:\Phi_6^7)|_{\Klein}:\Klein\to\PP^1$ is a morphism of degree $168$ whose fibres are exactly the $G$-orbits. Its fibres over $\infty$, $0$ and $1728$ are $7\,O_{24}$, $3\,O_{56}$ and $2\,O_{84}$, and every other fibre is a free orbit of $168$ distinct points.
\end{lemma}

\begin{proof}
$\Phi_6$ and $\Phi_{14}$ have no common zero on $\Klein$ since $O_{24}\cap O_{56}=\emptyset$, so $\pi$ is a morphism and its degree is $\deg(7\,O_{24})=168$. It is $G$-invariant, so it
factors through $\Klein\to\Klein/G\cong\PP^1$, which has degree $168$. The induced map $\Klein/G\to\PP^1$ has degree $1$, hence the fibres of $\pi$ are the orbits. The fibre over $1728$ is $\divisor(\Phi_{14}^3-1728\Phi_6^7)=\divisor(\Phi_{21}^2)=2\,O_{84}$ by
\eqref{eq:klein-relation}. By Riemann--Hurwitz,
$4=168\cdot(-2)+24\cdot6+56\cdot2+84\cdot1$, so $\pi$ has no further ramification.
\end{proof}

\begin{lemma}
\label{lem:invariance}
For every $n$ the restriction $\Omega_n|_{\Klein}$ is $G$-invariant. Moreover, $H^0(\Klein,\OO_{\Klein}(d))^G$ is spanned by the restrictions of
\[
  \begin{gathered}
  \Phi_6\ (d=6),\qquad \Phi_6\Phi_{21}\ (d=27),\qquad
  \Phi_6^{9}\Phi_{21},\ \Phi_6^{2}\Phi_{14}^{3}\Phi_{21}\ (d=75),\\
  \Phi_6^{21-7c}\Phi_{14}^{3c}\Phi_{21},\ 0\le c\le 3\ (d=147).
  \end{gathered}
\]
\end{lemma}

\begin{proof}
$G\subset\SL_3(\CC)$ fixes $\Phi_4$, so $\Omega_n\circ g\equiv\Omega_n$ modulo $\Phi_4$
by Corollary~\ref{cor:covariance} (for $n=4$ the same argument applies, since $g$ acts on $V_4=S_4/\Phi_4S_0$ with determinant $1$; alternatively, $G$ is simple, so every character of $G$ is trivial). Restriction $S_d\to H^0(\Klein,\OO_{\Klein}(d))$ is surjective and equivariant, so by averaging it is surjective on invariants, with kernel $\Phi_4\,S^G_{d-4}$. By \ref{K:ring},
$S^G/\Phi_4S^G\cong\CC[\Phi_6,\Phi_{14}]\oplus\Phi_{21}\CC[\Phi_6,\Phi_{14}]$, so a basis in degree $d$ is given by the monomials $\Phi_6^b\Phi_{14}^c\Phi_{21}^e$ with
$6b+14c+21e=d$ and $e\in\{0,1\}$; for odd $d$ we need $e=1$.
\end{proof}

\subsection{The Weierstrass divisors of the Klein quartic}

\begin{theorem}\label{thm:klein}
Put
\[ \Psi_{42}=3\,\Phi_{14}^3+21962\,\Phi_6^7\ \equiv\ 3\,\Phi_{21}^2+27146\,\Phi_6^7\pmod{\Phi_4}. \]
\begin{enumerate}[label=\rm(\roman*)]
\item $\W_1(\Klein)=O_{24}$: the $24$ flexes are ordinary.
\item $\Omega_2\equiv c\,\Phi_6\Phi_{21}\pmod{\Phi_4}$ and $\W_2(\Klein)=O_{24}+O_{84}$.
  In particular, the sextactic points of $\Klein$ are the $84$ fixed points of the involutions, each of $2$-weight $1$.
\item $\Omega_3\equiv c'\,\Phi_6^2\,\Phi_{21}\,\Psi_{42}\pmod{\Phi_4}$ and
  \[ \W_3(\Klein)=2\,O_{24}+O_{84}+O_{168},\]
  where $O_{168}=\Klein\cap\V(\Psi_{42})=\pi^{-1}\bigl(-\tfrac{21962}{3}\bigr)$ is a free $G$-orbit of $168$ dectactic points.
\end{enumerate}
Here $c,c'\in\CC^*$. Thus
\[
\begin{array}{lccc}
\toprule
\textup{orbit} & \textup{type} & \textup{vanishing sequence of }V_3 & 3\textup{-weight}\\
\midrule
O_{24}=\Klein\cap\V(\Phi_6) & \textup{ordinary flexes} & 0,1,\ldots,7,9,10 & 2\\
O_{84}=\Klein\cap\V(\Phi_{21}) & \textup{sextactic points} & 0,1,\ldots,8,10 & 1\\
O_{168}=\Klein\cap\V(\Psi_{42}) & \textup{dectactic points} & 0,1,\ldots,8,10 & 1\\
O_{56}=\Klein\cap\V(\Phi_{14}) & \textup{not $3$-Weierstrass} & 0,1,\ldots,9 & 0\\
\bottomrule
\end{array}
\]
and $\deg\W_3(\Klein)=48+84+168=300$, with $24+84+168=276$ distinct $3$-Weierstrass
points.
\end{theorem}

\begin{proof}
(i) $\divisor(\Phi_6|_\Klein)=O_{24}$ is reduced, so every flex has $w_1=1$, i.e.\ is ordinary.

(ii) By Lemma~\ref{lem:invariance}, $\Omega_2|_{\Klein}$ is a multiple of $\Phi_6\Phi_{21}|_{\Klein}$, and it is non-zero; hence $\W_2(\Klein)=O_{24}+O_{84}$. The points of $O_{84}$ are sextactic also by Theorem~\ref{thm:homologies}(a).

(iii) By Lemma~\ref{lem:invariance},
$\Omega_3|_{\Klein}=\Phi_{21}\Phi_6^2\,(\lambda\Phi_6^7+\mu\Phi_{14}^3)|_{\Klein}$, and
$\lambda,\mu$ can be taken in $\QQ$ because $\Omega_3$ and the invariants are defined over $\QQ$. To determine $\lambda:\mu$ we evaluate both sides exactly at the point $P_0=(2:\theta:1)\in\Klein$, where $\theta^3+8\theta+2=0$. Comparing coefficients in
the cubic field $\QQ(\theta)$ gives three linear equations for $(\lambda,\mu)$; they are consistent and yield $\lambda:\mu=21962:3$. Hence $\Omega_3\equiv c'\,\Phi_6^2\Phi_{21}\Psi_{42}$. (The script \sing{klein.sing} performs this evaluation and, independently, verifies the congruence of forms directly from the Hamiltonian formula.)

On $\Klein$, $\Psi_{42}=0$ means $\pi=-21962/3$, which is not one of $\infty,0,1728$.
By Lemma~\ref{lem:quotient}, $\Klein\cap\V(\Psi_{42})$ is a free orbit $O_{168}$ of $168$ distinct points, disjoint from $O_{24}\cup O_{56}\cup O_{84}$. With \ref{K:orbits}
this gives $\divisor(\Omega_3|_\Klein)=2\,O_{24}+O_{84}+O_{168}.$ The points of $O_{168}$ are not flexes and not sextactic by (ii), hence dectactic, and $\Omega_3$ does not vanish on $O_{56}$. The vanishing sequences follow from the weights: on $O_{24}$ Lemma~\ref{lem:flex} with $w_3=2$ forces $m=10$, and $\{0,\ldots,8,10\}$ is the only sequence of weight $1$.
\end{proof}

\begin{example}\label{ex:klein-flex}
The point $P=(0:0:1)\in O_{24}$ gives an independent check that $\mu\neq0$. In the chart $Z=1$ the curve is $x+y^3+x^3y=0$ with tangent line $x=0$. The cubic
$E:\ XZ^2+Y^3=0$, i.e.\ $x+y^3=0$, is smooth at $P$ and parametrised by $y$; substituting
$x=-y^3$ gives
\[ x+y^3+x^3y=(-y^3)+y^3+(-y^3)^3y=-y^{10},\]
so $I_P(\Klein,E)=10$. Thus $m=10$ in Lemma~\ref{lem:flex}, the vanishing sequence is
$\{0,\ldots,7,9,10\}$ and $w_3(P)=2$, in agreement with the factor $\Phi_6^2$. If we had $\mu=0$, the flexes would have weight $9$. Note that no homology of $\PP^2$ preserves $\Klein$ except the $21$ involutions (elements of order $3$, $4$ and $7$ of
$G\subset\SL_3$ have three distinct eigenvalues; \sing{groups.sing} generates $G$ and finds no homologies other than the $21$ involutions), so the weight $2$ at the flexes is not forced by
Theorem~\ref{thm:homologies}.
\end{example}

\begin{remark}\label{rem:naive}
By \eqref{eq:klein-relation}, $\Phi_6^2\Phi_{21}^3\equiv\Phi_{21}\Phi_6^2(\Phi_{14}^3-1728\Phi_6^7)$ is the member $\lambda:\mu=-1728:1$ of the same pencil. It is invariant of degree $75$, and it would give weights $2$ on $O_{24}$ and $3$ on $O_{84}$, again of total
$24\cdot2+84\cdot3=300$, and a parity-compatible vanishing sequence $\{0,\ldots,8,12\}$ at the axis points (Theorem~\ref{thm:KE}). The computation rules it out: the $84$ sextactic points carry exactly the weight forced by Theorem~\ref{thm:homologies}, and the remaining weight sits on the generic orbit $O_{168}$. The script \sing{klein.sing} confirms that $\Omega_3$ is not proportional to $\Phi_6^2\Phi_{21}^3$ modulo $\Phi_4$ and, via the transversality of $\V(\Phi_4,\Psi_{42})$ and its disjointness from $O_{24}\cup O_{56}\cup O_{84}$, that $\V(\Phi_4,\Omega_3)$ has exactly $276$ distinct points.
\end{remark}

\begin{remark}
On $O_{168}$ one has $\Phi_{14}^3/\Phi_6^7=-21962/3$, where $21962=2\cdot79\cdot139$. In the notation $J=\Phi_{14}^3/(1728\,\Phi_6^7)$, for which the
three branch values are $J=\infty,0,1$, the dectactic orbit is the fibre $J=-10981/2592$.
\end{remark}

\begin{proposition}\label{prop:klein4}
Modulo $\Phi_4$ and up to a non-zero scalar,
\[
\Omega_4\equiv\Phi_{21}\bigl(\Phi_{14}^9+188268\,\Phi_{14}^6\Phi_6^7
+2736969693\,\Phi_{14}^3\Phi_6^{14}+1770914087580\,\Phi_6^{21}\bigr).
\]
The cubic $p(s)=s^3+188268s^2+2736969693s+1770914087580$ is irreducible over $\QQ$ and has three distinct real roots $s_1\approx-172457.1$, $s_2\approx-15132.30$,
$s_3\approx-678.596$, none of them equal to $0$, $1728$ or $-21962/3$. Hence
\[ \W_4(\Klein)=O_{84}+\pi^{-1}(s_1)+\pi^{-1}(s_2)+\pi^{-1}(s_3), \]
$588$ distinct points of weight one. In particular neither the flexes nor the dectactic orbit $O_{168}$ are $4$-Weierstrass points: on $\Klein$ we have
$\W_1\subseteq\W_2\subseteq\W_3$, but $\W_1\not\subseteq\W_4$ and $\W_3\not\subseteq\W_4$.
\end{proposition}

\begin{proof}
By Lemma~\ref{lem:invariance}, $\Omega_4|_\Klein$ is a combination of the four invariants $\Phi_6^{21-7c}\Phi_{14}^{3c}\Phi_{21}$; \sing{klein.sing} computes $\Omega_4$ ($14\times14$ Hamiltonian determinant, under a minute) and verifies the displayed congruence. The roots of $p$ are the values of $\pi$ on the residual divisor, and Lemma~\ref{lem:quotient} applies. The script also checks directly that the residual divisor consists of $504$ distinct points disjoint from $O_{84}$, that neither $O_{24}$ nor $O_{168}$ meets $\V(\Omega_4)$, that $p$ is irreducible over $\QQ$ and non-zero at $0$, $1728$, $-21962/3$, and that $p$ changes sign on three disjoint intervals around $s_1,s_2,s_3$.
\end{proof}

\section{The Picard quartic revisited}\label{sec:picard}

Kamel and Elshareef computed the $3$-Weierstrass divisor of the Picard quartic
\[ \Pic:\ F=Y^3Z-X^4-Z^4=0\qquad(\text{affinely } y^3=x^4+1) \]
with their affine technique \cite[Section~4.2]{KamelElshareef}. Here we redo the computation with $\Omega_n$ and show that every flex and every sextactic point in $\W_3(\Pic)$ is explained by Theorem~\ref{thm:homologies}. Let $\omega=e^{2\pi i/3}$ and consider:
\begin{itemize}[leftmargin=2em]
\item $\sigma:(X:Y:Z)\mapsto(X:\omega Y:Z)$, a homology of order $3$ with axis $Y=0$ and centre $P_\infty=(0:1:0)\in\Pic$;
\item $\rho:(X:Y:Z)\mapsto(iX:Y:Z)$, a homology of order $4$ with axis $X=0$ and centre $(1:0:0)$; $\rho^2$ is an involution with the same axis;
\item $\tau_{k,\pm}:(X:Y:Z)\mapsto(\pm\sqrt3\,X:\ Y+2\omega^kZ:\ \omega^{-k}Y-Z)$, $k=0,1,2$. One checks $F\circ\tau_{k,\pm}=9F$ and $\tau_{k,\pm}^2=3\cdot\mathrm{id}$, so these are six involutions; their axes are the lines $Y=\omega^k(1\pm\sqrt3)Z$;
\item the conjugates $\tau_{k,+}\sigma\tau_{k,+}^{-1}$, homologies of order $3$ with centres $(0:\omega^k:1)$ and axes $Y=-2\omega^kZ$.
\end{itemize}
The elements $\sigma$, $\rho$, $\tau_{0,+}$ generate a group of order $48$ (the full automorphism group of $\Pic$, cf.\ \cite{KuribayashiKomiya,Dolgachev}), and this group contains exactly seven involutions, four cyclic groups of homologies of order three and one of order four (\sing{groups.sing}; the properties of $\sigma$, $\rho$, $\tau_{k,\pm}$ and of the conjugates are checked in \sing{picard.sing}). So $\Pic$ has four homologies of order three (up to inverses),
with centres the four points of $\Pic\cap\V(X)$ and axes the four lines $Y(Y^3+8Z^3)=0$, and seven involutions, with axes $X=0$ and the six lines $A_6=0$, where
\[ A_6=Y^6-20Y^3Z^3-8Z^6=\prod_{k=0}^{2}\bigl(Y-\omega^k(1+\sqrt3)Z\bigr)\bigl(Y-\omega^k(1-\sqrt3)Z\bigr). \]

\begin{theorem}\label{thm:picard}
Modulo $F$ and up to non-zero scalars,
\[\Omega_1\equiv X^2\,Y(Y^3+8Z^3),\qquad
  \Omega_2\equiv X^5\,Y(Y^3+8Z^3)\,A_6\,B_{12},\qquad
\Omega_3\equiv X^5\,Y^4(Y^3+8Z^3)^4\,A_6\,\Theta_{48},
\]
where $B_{12}=7Y^{12}+8360Y^9Z^3-23232Y^6Z^6+28160Y^3Z^9-8192Z^{12}$ and $\Theta_{48}=\theta(Y^3,Z^3)$ with
\begin{align*}
\theta(a,b)={}&a^{16}+73568a^{15}b+459631488a^{14}b^2-23704507136a^{13}b^3+43651171023104a^{12}b^4\\
&-525897401088000a^{11}b^5+2927751974158336a^{10}b^6-9709444779212800a^9b^7\\
&+21463663457599488a^8b^8-33764831584583680a^7b^9+39204280720162816a^6b^{10}\\
&-33925035064295424a^5b^{11}+21508446746574848a^4b^{12}-9592689196531712a^3b^{13}\\
&+2832204514197504a^2b^{14}-505775348776960ab^{15}+43980465111040b^{16},
\end{align*}
irreducible over $\QQ$. The $3$-Weierstrass divisor is
\[
\begin{array}{llcccl}
\toprule
\textup{cut out by} & \textup{type} & \# & \textup{sequence of }V_3 & w_3 & \textup{reason}\\
\midrule
X & \textup{hyperflexes} & 4 & 0,\ldots,6,8,9,12 & 5 & \textup{centres, order }3\\
Y(Y^3+8Z^3) & \textup{ordinary flexes} & 16 & 0,\ldots,7,9,12 & 4 & \textup{axes, order }3\\
A_6 & \textup{sextactic} & 24 & 0,\ldots,8,10 & 1 & \textup{axes, order }2\\
B_{12} & \textup{sextactic} & 48 & 0,\ldots,9 & 0 & \textup{---}\\
\Theta_{48} & \textup{dectactic} & 192 & 0,\ldots,8,10 & 1 & \textup{---}\\
\bottomrule
\end{array}
\]
with $4\cdot5+16\cdot4+24+192=300$ and $236$ distinct $3$-Weierstrass points, in accordance with \cite{KamelElshareef}. Each of the $16$ ordinary flexes $P$ admits a smooth cubic $E_P$ with $E_P\cdot\Pic=12P$; for $P=(\zeta:0:1)$, $\zeta^4=-1$,
\[  F=Z\cdot\bigl(Y^3-4\zeta X^3+6\zeta^2X^2Z-4\zeta^3XZ^2-2Z^3\bigr)-(X-\zeta Z)^4 .\]
\end{theorem}

\begin{proof}
The lower bounds are Theorem~\ref{thm:homologies}: the points of $\V(X)$ are hyperflexes (axis points of $\rho$, and centres of the order-three homologies), so $w_3=5$; the points on the four axes $Y(Y^3+8Z^3)=0$ are ordinary flexes with $w_3=4$;
the $24$ points on the six axes $A_6=0$ are non-flex points with $w_3\ge1$. These account for $20+64+24=108$ of the weight. The identities for $\Omega_1$, $\Omega_2$, $\Omega_3$, the transversality of all factors, their pairwise disjointness and the irreducibility of $\theta$ are verified in \sing{picard.sing}; the residual divisor of degree $192$ is reduced, so it consists of dectactic points of weight one, and the $48$ points of $\V(B_{12})$, which are $1$-sextactic by the formula for $\Omega_2$, are not $3$-Weierstrass points. The displayed identity is Theorem~\ref{thm:homologies}(b) for $u=Y$, $F_1=Z$, $F_4=-X^4-Z^4$, $t=X-\zeta Z$,
$\lambda=-1$.
\end{proof}

\begin{remark}[Relation to the affine formulas of \cite{KamelElshareef}]
All factors of $\Omega_2$ and $\Omega_3$ except the power of $X$ are binary forms in $Y,Z$. This is forced by $\rho$: using $X^4\equiv Y^3Z-Z^4$, every form that is an eigenvector of $\rho$ (as $\Omega_n$ is, by Corollary~\ref{cor:covariance}) is congruent modulo $F$ to $X^a$ times a binary form in $Y,Z$ for some $0\le a\le3$, and a binary form cuts out a union of fibres of the projection from the centre $(1:0:0)$ of $\rho$, i.e.\ of $\rho$-orbits. Since $\Omega_n$ is also an eigenvector of $\sigma$,
these binary forms are, up to a power of $Y$, polynomials in $Y^3$ and $Z^3$. On the affine curve $y^3=x^4+1$ (with $Z=1$) this explains why the polynomials of
\cite{KamelElshareef} are polynomials in $x^4$: $Y^3+8Z^3\mapsto x^4+9$, $A_6\mapsto Q_8=x^8-18x^4-27$, $B_{12}\mapsto Q_{16}$ with $Q_{16}=7x^{16}+8388x^{12}+1890x^8+6804x^4+5103$, and $\Theta_{48}\mapsto Q_{64}$, the polynomial of degree $64$ in \cite[Sec.~4.2]{KamelElshareef} (all checked in \sing{picard.sing}). The phenomenon noted in \cite[Remark~2]{KamelElshareef}, that sextactic points of the same type can have different $3$-weights, is explained by Theorem~\ref{thm:homologies}: the sextactic points in $\W_3(\Pic)$ are precisely those lying on axes of involutions, exactly as for the Komiya--Kuribayashi quartic.
\end{remark}

\begin{remark}
In order four, \sing{picard.sing} finds $4$-weight $5$ at the hyperflexes, $4$ at the ordinary flexes and $1$ at the $24$ involution axis points, in accordance with Theorem~\ref{thm:allq} (for an axis point of a homology of order three the residue
count gives $w_4\ge4$), while none of the $192$ dectactic points is a $4$-Weierstrass point.
\end{remark}

\section{Concluding remarks}\label{sec:comparison}

\begin{table}[ht]
\centering
\small
\renewcommand{\arraystretch}{1.2}
\resizebox{\textwidth}{!}{%
\begin{tabular}{@{}lcccc@{}}
\toprule
 & Fermat $\Fermat$ & Komiya--Kuribayashi $\KK$ & Klein $\Klein$ & Picard $\Pic$\\
\midrule
involutions / order-$3$ / order-$4$ homologies & $15/0/3$ & $9/0/0$ & $21/0/0$ & $7/4/1$\\
flexes & $12$ hyperflexes & $12$ hyperflexes & $24$ ordinary & $4$ hyper, $16$ ordinary\\
sextactic points & $48$ & $48$ & $84$ & $72$\\
$\Omega_2$ mod $F$ & $X^5Y^5Z^5\Delta_{12}$ & $L_+^2L_-XYZ\,S_6$ & $\Phi_6\Phi_{21}$ & $X^5Y(Y^3+8Z^3)A_6B_{12}$\\
$\Omega_3$ mod $F$ & $X^5Y^5Z^5\Delta_{12}\Pi_{48}$ & $L_+^2L_-XYZ\,\Psi_{54}$ & $\Phi_6^2\Phi_{21}\Psi_{42}$ & $X^5Y^4(Y^3+8Z^3)^4A_6\Pi_{48}$\\
flex contribution to $\W_3$ & $12\cdot 5$ & $12\cdot 5$ & $24\cdot 2$ & $4\cdot5+16\cdot4$\\
sextactic contribution & $48\cdot 1$ & $24\cdot 1$ & $84\cdot 1$ & $24\cdot1$\\
dectactic contribution & $192\cdot 1$ & $216\cdot 1$ & $168\cdot 1$ & $192\cdot1$\\
distinct $3$-Weierstrass points & $252$ & $252$ & $276$ & $236$\\
$\W_3\subseteq\W_4$? & no & no & no & no\\
\bottomrule
\end{tabular}}
\caption{Second- and third-order Weierstrass ivisors of the four quartics. The homologies are counted in $\Aut(\Fermat)$, $G_{24}\subseteq\Aut(\KK)$, $\Aut(\Klein)$
and $\Aut(\Pic)$, cyclic groups of homologies counted once.}
\label{tab:comparison}
\end{table}

Table~\ref{tab:comparison} summarises the examples. The local analysis of Section~\ref{sec:local} explains most of their structure. Hyperflexes always
contribute $5$ to the total weight $300$. An ordinary flex contributes $0$ (generically), $2$ (e.g.\ on the Klein quartic) or $4$ (exactly when $C$ lies in a pencil $\langle T_P^4,\ell\cdot E\rangle$ as in Proposition~\ref{prop:totalcontact};
e.g.\ on the axes of a homology of order three, as on the Picard quartic). Points on the axes of involutions always contribute, which accounts for all sextactic points of $\Fermat$ and $\Klein$ but only for some of those of $\KK$ and $\Pic$. The rest of the weight sits on dectactic points, which for these curves are simple and, in contrast to the symmetric points, do not persist in order four.

For order two, a single classical covariant, Cayley's $2$-Hessian in the form corrected by Maugesten and Moe \cite{MaugestenMoe}, detects sextactic points together with flex contributions. For order three the natural object is the Wronskian of the third Veronese system. On a smooth quartic it is a section of $\OO_C(75)$ with divisor of degree $300$, conceptually simple but computationally large. We conclude with the following problem.
\begin{problem}Is there an explicit covariant of plane quartics, a ``third Hessian'', whose restriction to $C$ is $\Omega_3$, analogous to Cayley's $2$-Hessian?
Proposition~\ref{prop:projective} gives the covariant $\det(J_L^iM_j)\equiv-L^{45}\Omega_3$ of the pair $(F,L)$; the problem is to remove the auxiliary line $L$ by a closed formula.
\end{problem}

\appendix
\section{The Singular scripts}\label{app:comp}

The ancillary files are seven \texttt{SINGULAR} scripts and the file \sing{omega.sing} of shared procedures, tested with \texttt{SINGULAR}~4.3.2 \cite{Singular}. Each script runs on its own, e.g.\ \sing{Singular -q fermat.sing}, and prints every check as a statement followed by \sing{1} (true) or \sing{0} (false); computed vanishing sequences are printed next to the expected ones. The file \sing{output.txt} contains the complete output; all scripts together take about two minutes, half of it for $\Omega_4$ of $\Klein$. All computational scripts and source code supporting the calculations and findings of this study are openly available on GitHub at \cite{github}.

\paragraph{Shared procedures.} \sing{Omega(F,n)} computes $\Omega_n$ by Proposition~\ref{prop:hamilton}: it dehomogenises the standard monomials of degree $n$ modulo $F$, applies $\delta$ repeatedly while reducing modulo $f$, takes the determinant, reduces it with respect to a degree ordering and homogenises to degree $d_n$. \sing{prop(A,B,F)} tests $A\equiv cB\pmod F$ with $c\neq0$; \sing{transversal(F,G)} tests that $\V(F,G)$ consists of $\deg F\cdot\deg G$ distinct points (no common zero of $F$, $G$ and the $2\times2$ minors of their Jacobian matrix); \sing{empty(I)} tests $\V(I)=\emptyset$; \sing{vanseq(F,n,x0,y0)} returns the vanishing sequence of $V_n$ at $(x_0:y_0:1)$ as in Remark~\ref{rem:recipe}, and \sing{wsum} its weight. Irrational points are handled over a number field, given by a ring parameter with a minimal polynomial; quotients such as $\Psi_{54}$ are obtained with \sing{lift}.

\paragraph{The scripts.}
\begin{center}\small
\begin{tabular}{@{}l l p{0.55\textwidth}@{}}
\toprule
script & for & checks\\
\midrule
\sing{general.sing} & Sections~\ref{sec:prelim}--\ref{sec:local} & degrees $d_n$; Proposition~\ref{prop:projective}, Corollary~\ref{cor:covariance}; vanishing sequences in the examples of \cite[Sec.~3]{KamelElshareef} and for homologies of orders $2,3,4$; residue counts of Theorem~\ref{thm:allq}\\
\sing{fermat.sing} & Section~\ref{sec:fermat} & Theorems~\ref{thm:fermat} and~\ref{thm:fermat4}, $P_{12}$, $P_{48}$, the check modulo $33037$\\
\sing{kk.sing} & Section~\ref{sec:kk} & Proposition~\ref{prop:kk-flexes}, Theorem~\ref{thm:kk}, order four\\
\sing{klein.sing} & Section~\ref{sec:klein} & invariants, \eqref{eq:klein-relation}, \ref{K:orbits}, Theorem~\ref{thm:klein}, Example~\ref{ex:klein-flex}, Remark~\ref{rem:naive}, Proposition~\ref{prop:klein4}\\
\sing{picard.sing} & Section~\ref{sec:picard} & homologies, Theorem~\ref{thm:picard} ($\Theta_{48}$ is \sing{Pi48}), $Q_8$, $Q_{16}$, $Q_{64}$, order four\\
\sing{generic.sing} & Remark~\ref{rem:general} & $C_0$ over $\QQ$ and modulo $32003$ and $1000003$\\
\sing{groups.sing} & Sections~\ref{sec:fermat}--\ref{sec:picard} & groups of orders $96$, $24$, $168$, $48$ and their homologies (non-scalar elements with minimal polynomial of degree two)\\
\bottomrule
\end{tabular}
\end{center}

\paragraph{Reduction modulo \texorpdfstring{$p$}{p}.}
Only Remark~\ref{rem:general}, the case $n=3$ in \sing{general.sing} and the independent check after Theorem~\ref{thm:fermat4} use reduction modulo a prime; everything else is exact over $\QQ$ or a number field. If $Z\subset\PP^2_{\ZZ_{(p)}}$ is the complete intersection $\V(F,\Omega)$ and both fibres are finite of the same degree, then $Z$ is finite and flat over $\ZZ_{(p)}$, so an étale special fibre forces an étale generic fibre, and a closed subscheme with empty special fibre has empty generic fibre. For a point $\tilde P$ of $C$ over a $p$-adic field with good reduction $P$, the dimension of the space of sections of $V_n$ vanishing to order $\ge k$ can only grow under reduction, so the vanishing sequence at $\tilde P$ is termwise bounded by that at $P$. Since \sing{vanseq} computes ranks of restriction maps and never divides by factorials, it is valid in every characteristic.

\paragraph{AI Disclaimer.} During the preparation of this manuscript, the author(s) used AI-assisted technologies for the purposes of improving the English phrasing and readability of the text, as well as assisting with the implementation and debugging of \texttt{SINGULAR} routines.

\bigskip

\noindent
Ewelina Nawara\\
Department of Mathematics,\\
University of the National Education Commission Krakow\\
Podchor\c a\.zych 2, PL-30-084 Krak\'ow, Poland.\\
\textit{E-mail address:} \texttt{ewelina.nawara@uken.krakow.pl}
\end{document}